\documentclass[12pt,reqno]{amsart}

\numberwithin{equation}{section}

\usepackage{amssymb}
\usepackage{amsthm}
\usepackage{fullpage} 
\usepackage{xcolor}
\usepackage{mathtools} 

\usepackage{hyperref} 
\definecolor{mycitecolor}{rgb}{1,0,0}
\definecolor{mylinkcolor}{rgb}{0.66,0,0} 
\definecolor{myurlcolor}{rgb}{0.33,0,0}
\hypersetup{colorlinks=true,citecolor=mycitecolor,linkcolor=mylinkcolor,urlcolor=myurlcolor,breaklinks=true}

\definecolor{labelkey}{rgb}{0,0,1} 
\definecolor{refkey}{rgb}{0,0,0.5} 

\theoremstyle{plain}
\newtheorem{theorem}{Theorem}[section]

\newtheorem{corollary}[theorem]{Corollary}

\newtheorem{lemma}[theorem]{Lemma}
\newtheorem{proposition}[theorem]{Proposition}

\theoremstyle{definition}
\newtheorem{definition}[theorem]{Definition}

\newtheorem{remark}[theorem]{Remark}

\newtheorem*{notation*}{Notation}

\def\vhi{\varphi}
\def\ldiv{\backslash}
\def\rdiv{/}
\def\m{^{-1}}

\def\sym#1{\mathrm{Sym}(#1)}
\def\inn#1{\mathrm{Inn}(#1)}
\def\id#1{\mathrm{id}_{#1}}
\def\lps#1{\mathrm{Psa_{\ell}}(#1)}
\def\atp#1{\mathrm{Atp}(#1)}
\def\aut#1{\mathrm{Aut}(#1)}
\def\ker#1{\mathrm{Ker}(#1)}
\def\img#1{\mathrm{Img}(#1)}
\def\rad#1{\mathrm{Rad}(#1)}
\def\mlt#1{\mathrm{Mlt}(#1)}
\def\inn#1{\mathrm{Inn}(#1)}
\def\asc#1{\mathrm{Asc}(#1)}
\def\com#1{\mathrm{Com}(#1)}
\def\nuc#1{\mathrm{Nuc}(#1)}

\def\csum#1#2{
	\mathchoice
		{\sum_{\mathclap{#1}}^{\mathclap{#2}}}		
		{\sum_{#1}^{#2}}							
		{ERROR}									
		{ERROR}									
}

\def\lcsum#1{\mathchoice{\sum_{\mathclap{#1}}}{\sum_{#1}}{ERROR}{ERROR}}

\def\mysum#1#2#3{\mathchoice{\sum_{\mathclap{#1}}(#2,#3)}{\sum_{#1}(#2,#3)}{ERROR}{ERROR}}

\begin{document}

\title{On abelian-by-cyclic Moufang loops}

\author{Ale\v{s} Dr\'{a}pal}
\address[Dr\'apal]{Department of Mathematics, Charles University, Sokolovsk\'a 83, 18675 Praha 8, Czech Republic}
\email[Dr\'apal]{drapal@karlin.mff.cuni.cz}

\author{Petr Vojt\v echovsk\'y}
\address[Vojt\v{e}chovsk\'y]{Dept.~of Mathematics, University of Denver, 2390 S.~York St., Denver, CO 80208, USA}
\email[Vojt\v{e}chovsk\'y]{petr@math.du.edu}

\thanks{A.~Dr\'apal supported by the INTER-EXCELLENCE project LTAUSA19070 of M\v{S}MT Czech Republic. P.~Vojt\v{e}chovsk\'y supported by the Simons Foundation Mathematics and Physical Sciences Collaboration Grant for Mathematicians no.~855097.}

\keywords{Abelian by cyclic Moufang loop, Moufang loop, conjugation in Moufang loops, Moufang permutation, solvability, congruence solvability}

\subjclass{20N05}

\begin{abstract}
We study abelian-by-cyclic Moufang loops. We construct all split $3$-divisible abelian-by-cyclic Moufang loops from so-called Moufang permutations on abelian groups $(X,+)$, which are permutations that deviate from an automorphism of $(X,+)$ by an alternating biadditive mapping (satisfying certain properties). More generally, we obtain additional abelian-by-cyclic Moufang loops from so-called construction pairs. As an aside, we show that in the Moufang loops $Q$ obtained from a construction pair on $(X,+)$ the abelian normal subgroup $(X,+)$ induces an abelian congruence of $Q$ if and only if $Q$ is a group.
\end{abstract}

\maketitle

\section{Introduction}\label{Sc:Introduction}

Moufang loops were introduced in $1935$ \cite{Moufang} and studied ever since. Yet there remain serious gaps in our understanding of basic structural concepts for Moufang loops. One of the reasons for the state of affairs is the lack of constructions.

In this paper we start a systematic approach to constructions of abelian-by-cyclic Moufang loops, that is, Moufang loops $Q$ possessing an abelian normal subgroup $X$ such that $Q/X$ is cyclic. The main results are Theorem \ref{Th:AbelianByCyclic} (a construction of many abelian-by-cyclic Moufang loops from so-called construction pairs), Theorem \ref{Th:MoufangPair} (every conjugation by $a$ in a Moufang loop restricts to a so-called Moufang permutation on $X$ and the multiplication on the subloop $\langle a^3\rangle X$ is just like the abstract multiplication formula of Theorem \ref{Th:AbelianByCyclic}), Theorem \ref{Th:HomImg} (all $3$-divisible abelian-by-cyclic Moufang loops are homomorphic images of the loops obtained from Theorem \ref{Th:AbelianByCyclic} with construction pairs induced by Moufang permutations), and Theorem \ref{Th:Split} (a characterization of split $3$-divisible abelian-by-cyclic Moufang loops).

\subsection{A brief overview of constructions of Moufang loops}

The motivating and most important examples of Moufang loops come from the multiplicative loops of nonzero octonions \cite{Baez,ConwaySmith,Moufang1933,SpringerVeldkamp}. These octonionic Moufang loops have additional strong structural properties not shared by all Moufang loops. For instance, in the $16$-element loop $\mathbb O_{16}$ of basic units of real octonions, every square associates with all elements (making $\mathbb O_{16}$ an \emph{extra loop} \cite{CheinRobinson,Fenyves}) and there is in fact a unique nonidentity square (making $\mathbb O_{16}$ a \emph{code loop} \cite{CheinGoodaire,Griess}). All nonassociative finite simple Moufang loops are obtained as central factors of the loop of unit elements in split octonion algebras over finite fields \cite{Liebeck,Paige,Zorn}.

Another interesting class of Moufang loops are Moufang $p$-loops. Moufang $p$-loops are centrally nilpotent (see \cite{GlaubermanII} for $p$ odd, \cite{GlaubermanWright} for $p=2$, and \cite{DrapalNilp} for a recent elementary proof of both cases). Using geometric considerations, Bol \cite{Bol} constructed a nonassociative commutative Moufang loop of order $3^4$ (also see \cite{KepkaNemec} for all commutative Moufang loops of order less than $3^6$). Later, Bruck \cite{BruckTrans} constructed a nonassociative Moufang loop of order $p^5$ for every prime $p$, and Nagy and Valsecchi \cite{NagyValsecchi} classified Moufang loops of order $p^5$ for all primes $p>3$. Using a computational approach to central extensions, all Moufang loops of orders $2^6$ and $3^4$ (resp. $3^5$) were classified in \cite{NagyVojtechovsky2007} (resp. \cite{SlatteryZenisek}). The aforementioned code loops are equivalently described as Moufang $2$-loops $Q$ possessing a central subloop $Z$ of order $2$ such that $Q/Z$ is an elementary abelian $2$-group, cf.~\cite{Aschbacher,Griess}. Code loops of order $\le 2^9$ were enumerated in \cite{OBrienVojtechovsky}. The constructions of Moufang $p$-loops are often of combinatorial character and they say little about Moufang loops that are not of prime power order.

Yet another source of constructions is connected with the (still open) question for which integers $n$ there exists a nonassociative Moufang loop of order $n$. Chein described all Moufang loops of order less than $32$ in \cite{CheinTrans} and all Moufang loops of order less than $64$ in \cite{CheinMem} (see also \cite{GoodaireEtal,LOOPS} for a catalog of all nonassociative Moufang loops of order less than $64$). The constructions that appear in Chein's classification of small Moufang loops include the well-known Chein double $M(G,2)$ of a group $G$ as well as several detailed variations on the Chein double. Leong and Rajah \cite{LeongRajah} proved that any Moufang loop of order $p^\alpha q_1^{\alpha_1}\cdots q_k^{\alpha_k}$ with $p < q_1 < \cdots < q_k$ odd primes and with $\alpha\le 3$ and $\alpha_i\le 2$ is a group, and similarly for the case $p>3$ and $\alpha=4$. The intermediate order factorizations between this lower bound (on exponents) and the upper bound furnished by nonassociative Moufang $p$-loops was investigated extensively by Rajah and various coauthors, cf., for instance, \cite{CheeRajah,Rajah2001}. By the very nature of the question, the constructions in this area of research are mostly \emph{ad hoc} since a single nonassociative example is required to settle the existence question for any given order.

As far as general constructions of Moufang loops are concerned, there are substantial results of Kinyon and Kunen on semidirect products in the context of extra loops \cite{KinyonKunen}. For (not necessarily extra) Moufang loops, the first step was a formula discovered by Gagola \cite{Gagola2014}, cf.~our \eqref{Eq:V4}, that gives a necessary condition for the existence of a semidirect product of a Moufang loop (which is normal in the product) and a cyclic group of order coprime to three. Sufficient conditions for the existence of such a semidirect product were then given by Dr\'apal \cite{Drapal2017}. The conditions of \cite{Drapal2017} are in the form of certain requirements on semiautomorphisms of the normal subloop and they are difficult to work with. The present paper was inspired by \cite{Drapal2017} but is independent of it.

Finally, we wish to mention the papers \cite{Gagola2012,GreerRaney,LeongRajah1998} that deal with semidirect products and/or split extensions of Moufang loops.

\subsection{A summary of main results}

The following definition is key to obtaining many abelian-by-cyclic Moufang loops:

\begin{definition}\label{Df:ConstructionPair}
Let $(X,+)$ be an abelian group, $g$ a permutation of $X$ and $\gamma:X\times X\to X$ a mapping. Then $(g,\gamma)$ is a \emph{construction pair} on $(X,+)$ if $\gamma$ is a symmetric alternating biadditive mapping,
\begin{equation}\label{Eq:PairDef}
    g^{-1}(g(x)+g(y)) =x+y + \gamma(x,y)+g^{-1}(\gamma(x,y)) + g^{-2}(\gamma(x,y))\tag{C1}
\end{equation}
holds for all $x,y\in X$,
\begin{equation}\label{Eq:PairRad}
    \gamma(\gamma(x,y),z)=0\tag{C2}
\end{equation}
holds for all $x,y,z\in X$, and
\begin{equation}\label{Eq:PairProp}
    g^{-1}(\gamma(x,y)) = \gamma(g(x),y)\tag{C3}
\end{equation}
holds for all $x,y\in X$.
\end{definition}

Note that the condition \eqref{Eq:PairRad} may be equivalently expressed by stating that the image of $\gamma$ is contained in the radical $\rad{\gamma} = \{x\in X:\gamma(x,y)=0$ for all $y\in X\}$ of $\gamma$.

For $i,j\in\mathbb Z$, define the interval $I(i,j)$ of $\mathbb Z$ by
\begin{equation}\label{Eq:Interval}
    I(i,j) = \left\{\begin{array}{ll}
        \emptyset,&\text{ if }i=j,\\
        \{i,i+1,\dots,j-1\},&\text{ if }i<j,\\
        \{j,j+1,\dots,i-1\},&\text{ if }j<i.
    \end{array}\right.
\end{equation}
Given an abelian group $(X,+)$ and a cyclic group $C=\langle b\rangle$, we show that the formula
\begin{equation}\label{Eq:MultGeneralIntro}
    (b^i,x)\cdot(b^j,y) = \Big(b^{i+j},g^{-j}(x)+y+\lcsum{k\in I(i+j,-j)}g^{-k}(\gamma(x,y))\Big)
\end{equation}
correctly defines a multiplication on $C\times X$ if and only if either $C$ is infinite, or $C$ is finite, $|g|$ divides $|C|$ and $\sum_{0\le k<|C|}g^k(x)\in\rad{\gamma}$ for all $x\in X$. (The condition $\sum_{0\le k<|C|}g^k(x)\in\rad{\gamma}$ is satisfied whenever $|C|$ is finite, $|g|$ divides $|C|$ and $\rad{\gamma}$ contains no elements of order $3$.) In that case the resulting groupoid is in fact a Moufang loop $Q = C\ltimes_{(g,\gamma)}X$ that contains a normal subloop $1\times X$ such that $Q/(1\times X)$ is isomorphic to $C$, cf.~Theorem \ref{Th:AbelianByCyclic}. In addition, $(C\times 0)\cap (1\times X)=1$, so $Q$ is a split extension of $X$ by $C$.

Several key properties of the loops $C\ltimes_{(g,\gamma)}X$ are established in Section \ref{Sc:PropertiesCP}, such as the formula for an elementwise commutator, elementwise associator, the associator subloop, the commutator subloop, the derived subloop, the nucleus and the center.

In the brief Section \ref{Sc:Abelian} we show that in the Moufang loop $Q=C\ltimes_{(g,\gamma)}X$ the normal subloop $1\times X$ induces an abelian congruence of $Q$ if and only if $Q$ is a group. This is related to another current line of investigation of ours concerned with the two notions of solvability in loops, cf.~\cite{DrVo,StaVojComm,StaVojAbel}. The results of Section \ref{Sc:Abelian} are not used elsewhere in the paper.

\medskip

The rest of the paper is concerned with a partial converse of Theorem \ref{Th:AbelianByCyclic}.

It is easy to see that Theorem \ref{Th:AbelianByCyclic} does not yield all abelian-by-cyclic Moufang loops. For instance, the quaternion group $Q_8$ is abelian-by-cyclic but it is not split and hence it cannot be obtained from the construction of Theorem \ref{Th:AbelianByCyclic}. (A smallest nonassociative abelian-by-cyclic Moufang loop that is not split is of order $32$.) In fact, Theorem \ref{Th:AbelianByCyclic} does not yield all split abelian-by-cyclic Moufang loops either. For instance, the nonassociative commutative Moufang loop of order $81$ and exponent $3$ is a split extension of $X=\mathbb Z_3\times\mathbb Z_3\times\mathbb Z_3$ by $C=\mathbb Z_3$, but Theorem \ref{Th:AbelianByCyclic} does not yield any nonassociative commutative Moufang loops, cf.~Corollary \ref{Cr:NoCML}.

On the other hand, Theorem \ref{Th:AbelianByCyclic} yields all abelian-by-cyclic Moufang loops that are split and $3$-divisible, cf.~Theorem \ref{Th:Split}. To prove this fact, we will need the following notion similar to construction pairs:

\begin{definition}\label{Df:MoufangPerm}
Let $(X,+)$ be an abelian group. A permutation $f$ on $X$ is said to be a \emph{Moufang permutation} on $(X,+)$ if the mapping $\beta:X\times X\to X$ defined by
\begin{equation}\label{Eq:BetaDef}
    \beta(x,y) = f^{-1}(f(x)+f(y))-x-y\tag{P1}
\end{equation}
is (symmetric) alternating and biadditive,
\begin{equation}\label{Eq:BetaRad}
    \beta(\beta(x,y),z)=0\tag{P2}
\end{equation}
holds for all $x,y,z\in X$, and
\begin{equation}\label{Eq:BetaProp}
    \beta(f(x),f(y)) = f(\beta(f^3(x),y))\tag{P3}
\end{equation}
holds for all $x,y\in X$. When $f$ is a Moufang permutation and $\beta$ is defined by \eqref{Eq:BetaDef}, we call $\beta$ the \emph{biadditive mapping associated with}$ f$ and the tuple $(f,\beta)$ a \emph{Moufang pair}.
\end{definition}

The relation between construction pairs and Moufang pairs is that if $(f,\beta)$ is a Moufang pair then $(f^3,\beta)$ is a construction pair, cf.~Proposition \ref{Pr:MPCP}. In particular, when the remaining assumptions of Theorem \ref{Th:AbelianByCyclic} are satisfied, we can use the multiplication formula \eqref{Eq:MultGeneralIntro} with $(g,\gamma) = (f^3,\beta)$. If also $b=a^3$ for some $a\in C$ (which certainly holds when $C$ is $3$-divisible), the formula \eqref{Eq:MultGeneralIntro} becomes
\begin{equation}\label{Eq:MultMP}
    (a^{3i},x)\cdot(a^{3j},y) = \Big(a^{3(i+j)},f^{-3j}(x)+y+\lcsum{k\in I(i+j,-j)}f^{-3k}(\beta(x,y))\Big).
\end{equation}

After deriving some rather general results on pseudoautomorphisms and semiautomorphisms of Moufang loops induced by inner mappings in Sections \ref{Sc:Pseudo} and \ref{Sc:Semi}, we prove that for every abelian normal subgroup $X$ of a Moufang loop $Q$ and every element $a\in Q$ the restriction of the ``conjugation'' $T_a$ to $X$ is a Moufang permutation on $X$, cf.~Proposition \ref{Pr:Beta}.

In Section \ref{Sc:MPP} we study abstract properties of Moufang permutations $f$ and their powers. Among other results, we derive a formula for the expression $f^{-i}(f^i(x)+f^i(y))$, cf.~Proposition \ref{Pr:iBetaDef}.

We then prove in Theorem \ref{Th:MoufangPair} that if $(X,+)$ is an abelian normal subgroup of a Moufang loop $Q$ and if $a\in Q$ then
\begin{equation}\label{Eq:MultInner}
	a^{3i}x\cdot a^{3j}y = a^{3(i+j)}\Big(f^{-3j}(x)+y+\lcsum{k\in I(i+j,-j)}f^{-3k}(\beta(x,y))\Big)
\end{equation}
holds for all $i,j\in\mathbb Z$ and $x,y\in X$, where $f$ is the restriction of $T_a$ to $X$, which we know from Proposition \ref{Pr:Beta} is a Moufang permutation on $(X,+)$. Here and throughout the paper, if $(X,+)$ is a normal subloop of $(Q,\cdot)$, we write either $x+y$ or $x\cdot y$ for the product of $x,y\in X$ in $Q$.

Note that \eqref{Eq:MultInner} does not necessarily describe the multiplication for \emph{all} pairs of elements of $Q$. Nevertheless, it easily follows, cf.~Corollary \ref{Cr:MultMP}, that if $(X,+)\unlhd Q$ and $Q=\langle a^3\rangle X$ for some $a\in Q$, then \eqref{Eq:MultInner} does fully describe the multiplication of $Q$.

It is then not difficult to show that all $3$-divisible abelian-by-cyclic Moufang loops $Q=CX$ are homomorphic images of the loops $C\ltimes_{(f^3,\beta)}X$, where $(f,\beta)$ is a Moufang pair on $(X,+)$. Here, $|f^3|$ must divide $|C|$, else $C\ltimes_{(f^3,\beta)}X$ is not even defined. The kernels of the homomorphisms are described in Proposition \ref{Pr:Kernel}. Finally, $3$-divisible abelian-by-cyclic Moufang loops $Q=CX$ that are split are precisely the the loops $C\ltimes_{(f^3,\beta)}X$ with $(f,\beta)$ a Moufang pair on $(X,+)$, cf.~Theorem \ref{Th:Split}.

\medskip

Let us conclude the summary of main results with a few comments.

The multiplication formula \eqref{Eq:MultInner} might appear rather complicated and unnatural. A point of departure in deriving the formula is the important identity
\begin{displaymath}
	a^{3i}x\cdot ya^{3j} = a^{3(i+j)}T_a^{-i-2j}(T_a^{i-j}(x)T_a^{i-j}(y))
\end{displaymath}
valid in all Moufang loops. We record this identity and all its variants in Proposition \ref{Pr:TVariations}, including the already mentioned identity \eqref{Eq:V4} of Gagola \cite{Gagola2014}.

It turns out that in the finite case each construction pair $(g,\gamma)$ on $(X,+)$ induces an integer $m$ and a mapping $U\times U \to V$ that is alternating and bilinear over $R$, where $U = X/\rad{\gamma}$, $V = \langle \img{\gamma}\rangle$  and $R = \mathbb F_2[x]/(x^m-1)$. To describe construction pairs for a given abelian group $(X,+)$ thus means to start from a classification of the relevant alternating bilinear mappings and then include several steps of lifting. We intend to study the problem of obtaining construction pairs and of classifying the resulting Moufang loops up to isomorphism in a future paper.

\section{Background}\label{Sc:Background}

\subsection{Mappings}\label{Ss:Mappings}

Let $(X,+)$ be an abelian group. Let us call a mapping $\gamma:X\times X\to X$ \emph{symmetric} if $\gamma(x,y)=\gamma(y,x)$ for all $x,y\in X$, \emph{antisymmetric} if $\gamma(x,y)=-\gamma(y,x)$ for all $x,y\in X$, \emph{alternating} if $\gamma(x,x)=0$ for all $x\in X$, and \emph{biadditive} if $\gamma(x+y,z) = \gamma(x,z)+\gamma(y,z)$ and $\gamma(x,y+z)=\gamma(x,y)+\gamma(x,z)$ for all $x,y,z\in X$.

Note that a biadditive mapping satisfies $\gamma(0,x)=0=\gamma(x,0)$, $\gamma(-x,y)=-\gamma(x,y)=\gamma(x,-y)$, $\gamma(x-y,z) = \gamma(x,z)-\gamma(y,z)$ and $\gamma(x,y-z)=\gamma(x,y)-\gamma(x,z)$ for all $x,y,z\in X$. Indeed, $\gamma(0,x) = \gamma(0+0,x) = 2\gamma(0,x)$ yields $\gamma(0,x)=0$, then $0 = \gamma(x+(-x),y) = \gamma(x,y)+\gamma(-x,y)$ implies $\gamma(-x,y)=-\gamma(x,y)$, and finally $\gamma(x-y,z) = \gamma(x,z)+\gamma(-y,z) = \gamma(x,z)-\gamma(y,z)$

An alternating biadditive mapping is antisymmetric since $0=\gamma(x+y,x+y)=\gamma(x,x)+\gamma(x,y)+\gamma(y,x)+\gamma(y,y) = \gamma(x,y)+\gamma(y,x)$. A symmetric alternating biadditive mapping then satisfies $2\gamma(x,y)=\gamma(2x,y)=0$ because $\gamma(x,y)=\gamma(y,x)=-\gamma(x,y)$. All these properties will be used without reference throughout the paper.

The image of any mapping $g$ will be denoted by $\img{g}$. As in the introduction, for a symmetric mapping $\gamma:X\times X\to X$, the \emph{radical} of $\gamma$ is defined by
\begin{displaymath}
	\rad{\gamma} = \{x\in X:\gamma(x,y)=0\text{ for all }y\in X\}.
\end{displaymath}
If $\gamma$ is symmetric and biadditive, $\rad{\gamma}$ is a subgroup of $(X,+)$.

\begin{notation*}
In order to improve legibility, we will often omit parentheses around arguments of unary mappings. For instance, if $g:X\to X$ and $\gamma:X\times X\to X$, we might write $g\gamma(x,y)$ instead of $g(\gamma(x,y))$, and $\gamma(gx,y)$ instead of $\gamma(g(x),y)$.
\end{notation*}

\subsection{Loops}

See \cite{BruckTrans, Pflugfelder} for an introduction to loop theory.

A set $Q$ with a binary operation $\cdot$ and an element $1\in Q$ is a \emph{loop} if for every $x\in Q$ we have $x\cdot 1=1\cdot x = x$ and the translations
\begin{displaymath}
	L_x:Q\to Q,\,L_x(y)=x\cdot y\quad\text{and}\quad R_x:Q\to Q,\,R_x(y)=y\cdot x
\end{displaymath}
are permutations of $Q$. The induced division operations will be denoted by $x\ldiv y = L_x^{-1}(y)$ and $x\rdiv y = R_y^{-1}(x)$.

Associative subloops will be referred to as \emph{subgroups}. A loop $Q$ is \emph{power associative} if every element of $Q$ generates a subgroup. A loop $Q$ is \emph{diassociative} if every two elements of $Q$ generate a subgroup.

\begin{notation*}
We will often write $xy$ instead of $x\cdot y$ and use $\cdot$ to indicate priority of multiplication, e.g., $x\cdot yz$ stands for $x\cdot(y\cdot z)$.
\end{notation*}

The \emph{multiplication group} of a loop $Q$ is the permutation group $\mlt{Q}=\langle L_x,R_x:x\in Q\rangle$, and the \emph{inner mapping group} $\inn{Q}$ of $Q$ is the stabilizer of $1$ in $\mlt{Q}$. It is well known that $\inn{Q}=\langle T_x,L_{x,y},R_{x,y}:x,y\in Q\rangle$, where
\begin{displaymath}
	T_x = R_x^{-1}L_x,\quad L_{x,y} = L_{xy}^{-1}L_xL_y\quad\text{and}\quad R_{x,y} = R_{xy}^{-1}R_yR_x.
\end{displaymath}
We refer to the inner mappings $T_x$ as \emph{conjugations}.

The \emph{nucleus} $\nuc{Q}$ of $Q$ consists of all $x\in Q$ such that $x(yz)=(xy)z$, $y(xz)=(yx)z$ and $y(zx) = (yz)x$ for all $y,z\in Q$. The \emph{center} $Z(Q)$ of $Q$ consists of all $x\in\nuc{Q}$ such that $xy=yx$ for all $y\in Q$. The \emph{commutator} $[x,y]$ of $x,y\in Q$ is defined by $(yx)[x,y] = xy$. The \emph{associator} $[x,y,z]$ of $x,y,z\in Q$ is defined by $(x\cdot yz)[x,y,z] = xy\cdot z$.

\subsection{Morphisms and topisms}

Denote by $\sym{Q}$ the symmetric group on $Q$. A tuple $(c,f)\in Q\times\sym{Q}$ is a \emph{(left) pseudoautomorphism} of $Q$ if
\begin{equation}\label{Eq:DfPseudo}
    cf(x)\cdot f(y) = cf(xy)
\end{equation}
holds for every $x,y\in Q$. The element $c$ is called a \emph{companion} of $f$. The set of all pseudoautomorphisms of $Q$ forms a group $\lps{Q}$ under the operations
\begin{equation}\label{Eq:LPs}
    (c,f)(d,g) = (cf(d),fc)\quad\text{and}\quad (c,f)^{-1} = (f^{-1}(c\ldiv 1),f^{-1}).
\end{equation}

A permutation $f\in\sym{Q}$ is a \emph{semiautomorphism} of $Q$ if $f(1)=1$ and
\begin{equation}\label{Eq:Semi}
    f(x\cdot yx)= f(x)\cdot f(y)f(x)
\end{equation}
holds for all $x,y\in Q$. If $f$ is a semiautomorphism of a power associative loop $Q$ then $f(x^i) = f(x)^i$ for every $i\in\mathbb Z$.

The condition \eqref{Eq:Semi} can be written as $fL_xR_x = L_{f(x)}R_{f(x)}f$ and therefore also as
\begin{equation}\label{Eq:SemiEquiv}
    L_{f(x)}\m f L_x = R_{f(x)} f R_x\m.
\end{equation}
Thus any permutation $f$ of a loop $Q$ that satisfies $f(1)=1$ and \eqref{Eq:SemiEquiv} is a semiautomorphism.

An \emph{autotopism} of a loop $Q$ is a triple $(\alpha,\beta,\gamma)$ of permutations of $Q$ such that
\begin{displaymath}
    \alpha(x)\beta(y) = \gamma(xy)
\end{displaymath}
holds for all $x,y\in Q$. Autotopisms of $Q$ may be composed and inverted componentwise, forming the \emph{autotopism group} $\atp{Q}$.

Note that $(c,f)$ is a pseudoautomorphism of $Q$ if and only if $(L_c f, f,L_c f)$ is an autotopism of $Q$. We use this fact in the proof of the following observation about autotopisms in which the middle component fixes the identity element.

\begin{lemma}\label{Lm:AtpToPseudo}
Let $(Q,\cdot,1)$ be a loop. Suppose that $(\alpha,\beta,\gamma)\in \atp{Q}$ and $\beta(1) = 1$. Then $(\alpha(1),\beta)\in \lps{Q}$.
\end{lemma}
\begin{proof}
Let $c=\alpha(1)$. By our assumption, $\alpha(x) = \alpha(x)\beta(1) = \gamma(x\cdot 1) = \gamma(x)$ and $c\beta(x) = \alpha(1)\beta(x) = \gamma(1\cdot x) = \gamma(x)$ for every $x\in Q$. Hence $(\alpha,\beta,\gamma) = (L_c\beta,\beta,L_c\beta)$.
\end{proof}

\subsection{Moufang loops}

A loop $Q$ is \emph{Moufang} if it satisfies any one of the equivalent \emph{Moufang identities}
\begin{align}
    xy\cdot zx &= (x\cdot yz)x,\label{Eq:M1}\tag{M1}\\
    xy\cdot zx &= x(yz\cdot x),\label{Eq:M2}\tag{M2}\\
    x(y\cdot zy) &= (xy\cdot z)y,\label{Eq:M3}\tag{M3}\\
    x(y\cdot xz) &= (xy\cdot x)z.\label{Eq:M4}\tag{M4}
\end{align}

Let us summarize a few facts about Moufang loops, consisting of easy observations and standard results on Moufang loops \cite{BruckTrans}.

The identity \eqref{Eq:M1} can be stated as $(L_x,R_x,R_xL_x)\in\atp{Q}$. This is a useful characterization of Moufang loops in terms of autotopisms.

By Moufang Theorem \cite{DrapalMouf, Moufang}, if three elements $x$, $y$ and $z$ of a Moufang loop associate, that is, $x(yz)=(xy)z$, then the subloop $\langle x,y,z\rangle$ is a group. Consequently, Moufang loops are diassociative, power associative, satisfy the \emph{flexible law} $x(yx)=(xy)x$, the \emph{inverse properties} $x^{-1}(xy)=y=(yx)x^{-1}$, and so on. We take advantage of diassociativity in Moufang loops and write unambiguously $xyx$, $x^y = y^{-1}xy$, $[x,y]=x^{-1}y^{-1}xy$, etc.

\begin{lemma}\cite{DrVo}
Let $Q$ be a Moufang loop. Then
\begin{align}
	x^{-1}(xy\cdot z) &= yx^{-1}\cdot xz,\label{Eq:M5}\\
	 (z\cdot yx)x^{-1} &= zx\cdot x^{-1}y\label{Eq:M6}
\end{align}
for every $x,y,z\in Q$.
\end{lemma}

All inner mappings of a Moufang loop can be seen as pseudoautomorphisms, with suitable companions. In particular,
\begin{equation}\label{Eq:InnPseudo}
    (x^{-3},T_x),\quad ([y,x],R_{x,y})\quad\text{and}\quad ([x^{-1},y^{-1}],L_{x,y})
\end{equation}
are pseudoautomorphisms in a Moufang loop. Every pseudoautomorphism of a Moufang loop is a semiautomorphism.

\begin{lemma}\label{Lm:MoufPseudo}
Let $Q$ be a Moufang loop, $c\in Q$ and $f\in\sym{Q}$. Then $(c,f) \in \lps{Q}$ if and only if
\begin{displaymath}
    xc^{-1}\cdot cy = f(f^{-1}(x)f^{-1}(y))
\end{displaymath}
for all $x,y \in Q$.
\end{lemma}
\begin{proof}
Substituting $f^{-1}(x)$ for $x$ and $f^{-1}(y)$ for $y$ into $cf(x)\cdot f(y) = cf(xy)$ yields $cx\cdot y = cf(f^{-1}(x)f^{-1}(y))$. We are done upon multiplying by $c^{-1}$ on the left and applying \eqref{Eq:M5}.
\end{proof}

\begin{corollary}\label{Cr:T}
Let $Q$ be a Moufang loop. Then
\begin{displaymath}
    xa^{-3} \cdot a^3y = T_a^{-1}(T_a(x)T_a(y))
\end{displaymath}
for all $a,x,y \in Q$.
\end{corollary}

The equation \eqref{Eq:TKey} below is of crucial importance for this paper.

\begin{proposition}\label{Pr:TKey}
Let $Q$ be a Moufang loop. Then
\begin{equation}\label{Eq:TKey}
    a^{3i} x \cdot ya^{3j} = a^{3i} \cdot  T_a^{j-i}(T_a^{i-j}(x) T_a^{i-j}(y)) \cdot a^{3j}
\end{equation}
for all $a,x,y\in Q$ and all $i,j\in\mathbb Z$.
\end{proposition}
\begin{proof}
Let $b=a^{3i}$ and $\delta=i-j$. By diassociativity and \eqref{Eq:M1}, $b^ix\cdot yb^j = b^j(b^\delta x)\cdot yb^j = b^j(b^\delta x\cdot y)b^j$. By \eqref{Eq:M5} and Corollary \ref{Cr:T},
\begin{displaymath}
    b^\delta x\cdot y = b^{\delta}(xb^{-\delta}\cdot b^\delta y) = b^\delta T_{a^{-\delta}}(T_{a^\delta}(x)T_{a^\delta}(y)) =  b^\delta T_a^{-\delta}(T_a^{\delta}(x)T_a^\delta(y)),
\end{displaymath}
where we have used $T_{a^k} = T_a^k$ in the last step. Combining, we get
\begin{displaymath}
    b^ix \cdot yb^j = b^j(b^\delta x\cdot y)b^j = b^j\cdot b^\delta T_a^{-\delta}(T_a^{\delta}(x)T_a^\delta(y)) \cdot b^j = b^i\cdot T_a^{-\delta}(T_a^{\delta}(x)T_a^\delta(y))\cdot b^j.\qedhere
\end{displaymath}
\end{proof}

Here are some variations on the identity \eqref{Eq:TKey}:

\begin{proposition}\label{Pr:TVariations}
Let $Q$ be a Moufang loop. Then
\begin{alignat}{3} 
    &a^{3i}x\cdot a^{3j}y &&= a^{3(i+j)}T_a^{-i-2j}(T_a^{i-j}(x)T_a^{i+2j}(y)) &&= T_a^{2i+j}(T_a^{i-j}(x)T_a^{i+2j}(y))a^{3(i+j)},\label{Eq:V1}\\
    &a^{3i}x\cdot ya^{3j} &&= a^{3(i+j)}T_a^{-i-2j}(T_a^{i-j}(x)T_a^{i-j}(y)) &&= T_a^{2i+j}(T_a^{i-j}(x)T_a^{i-j}(y))a^{3(i+j)},\label{Eq:V2}\\
    &xa^{3i}\cdot a^{3j}y &&= a^{3(i+j)}T_a^{-i-2j}(T_a^{-2i-j}(x)T_a^{i+2j}(y)) &&= T_a^{2i+j}(T_a^{-2i-j}(x)T_a^{i+2j}(y))a^{3(i+j)},\label{Eq:V3}\\
    &xa^{3i}\cdot ya^{3j} &&= a^{3(i+j)}T_a^{-i-2j}(T_a^{-2i-j}(x)T_a^{i-j}(y)) &&= T_a^{2i+j}(T_a^{-2i-j}(x)T_a^{i-j}(y))a^{3(i+j)}\label{Eq:V4}
\end{alignat}
for all $a,x,y\in Q$ and all $i,j\in\mathbb Z$.
\end{proposition}
\begin{proof}
Note that $a^kz = a^kza^{-k}a^k = T_a^k(z)a^k$, so it suffices to establish the first equality in each \eqref{Eq:V1}--\eqref{Eq:V4}. With $u=T_a^{j-i}(T_a^{i-j}(x) T_a^{i-j}(y))$, equation \eqref{Eq:TKey} yields
\begin{displaymath}
    a^{3i}x\cdot ya^{3j} = a^{3i}ua^{3j} = T_a^{3i}(u)a^{3i}a^{3j} = T_a^{2i+j}(T_a^{i-j}(x) T_a^{i-j}(y))a^{3(i+j)},
\end{displaymath}
which is \eqref{Eq:V2}. The remaining identities then follow from \eqref{Eq:V2} as $a^{3i}x\cdot a^{3j}y = a^{3i}x\cdot T_a^{3j}(y)a^{3j}$, $xa^{3i}\cdot a^{3j}y = a^{3i}T_a^{-3i}(x)\cdot T_a^{3j}(y)a^{3j}$ and $xa^{3i}\cdot ya^{3j} = a^{3i}T_a^{-3i}(x)\cdot ya^{3j}$.
\end{proof}

\section{Construction pairs and their properties}\label{Sc:CPP}

Construction pairs were defined in the Introduction, cf.~Definition \ref{Df:ConstructionPair}. In this section we will establish basic properties of construction pairs $(g,\gamma)$ on $(X,+)$. The condition \eqref{Eq:PairDef} will be often used in the form
\begin{displaymath}
    g(x)+g(y) = g(x+y+\gamma(x,y)+g^{-1}(\gamma(x,y)) + g^{-2}(\gamma(x,y))).
\end{displaymath}

\begin{lemma}\label{Lm:PairRadImg}
Let $(g,\gamma)$ be a construction pair on $(X,+)$ and let $i$ be an integer. Then:
\begin{enumerate}
\item[(i)] $g(0)=0$, $g(2x)=2g(x)$ and $g(-x)=-g(x)$ for all $x\in X$,
\item[(ii)] $\img{\gamma}\subseteq\rad{\gamma}$, $2X\subseteq\rad{\gamma}$ and $2\img{\gamma}=0$,
\item[(iii)] $g^i$ permutes both $\img{\gamma}$ and $\rad{\gamma}$,
\item[(iv)] $g^i(x+y) = g^i(x)+g^i(y)$ whenever $\{x,y\}\cap\rad{\gamma}\ne\emptyset$,
\item[(v)] $g^i$ restricts to an automorphism of $\rad{\gamma}$.
\end{enumerate}
\end{lemma}
\begin{proof}
(i) We have $\gamma(x,0)=0$ thanks to biadditivity. Hence $g^{-1}(0) = g^{-1}\gamma(0,0) = \gamma(g(0),0) = 0$ by \eqref{Eq:PairProp} and $g(0)=0$ follows. Since $\gamma$ is alternating and $g(0)=0$, we then have $0 = \gamma(x,x)+g^{-1}\gamma(x,x) + g^{-2}\gamma(x,x) = g^{-1}(g(x)+g(x)) - x-x = g^{-1}(2g(x)) - 2x$ by \eqref{Eq:PairDef}, which yields $2g(x)=g(2x)$. Finally, $\gamma(x,-x)=-\gamma(x,x)=0$, so $0=\gamma(x,-x)+g^{-1}\gamma(x,-x)+g^{-2}\gamma(x,-x) = g^{-1}(g(x)+g(-x))$ by \eqref{Eq:PairDef}, and hence $g(x)+g(-x)=0$.

(ii) The condition $\img{\gamma}\subseteq\rad{\gamma}$ is a restatement of \eqref{Eq:PairRad}. Any symmetric alternating biadditive mapping satisfies $0=\gamma(2x,y) = 2\gamma(x,y)$.

(iii) We have $x\in\rad{\gamma}$ iff $\gamma(x,y)=0$ for all $y\in X$ iff $g^{-1}\gamma(x,y) = g^{-1}(0)=0$ for all $y\in X$ iff $\gamma(g(x),y)=0$ for all $y\in X$ by \eqref{Eq:PairProp} iff $g(x)\in\rad{\gamma}$. Similarly, $z\in\img{\gamma}$ iff $z = \gamma(x,y)$ for some $x,y\in X$ iff $g^{-1}(z) = g^{-1}\gamma(x,y)=\gamma(g(x),y)$ for some $x,y\in X$ by \eqref{Eq:PairProp} iff $g^{-1}(z)\in\img{\gamma}$.

(iv) Suppose without loss of generality that $x\in\rad{\gamma}$. Then $\gamma(x,y)=0$ and, by (i), $g^i\gamma(x,y)=0$ for all $i$. Thus $0 = \gamma(x,y)+g^{-1}\gamma(x,y)+g^{-2}\gamma(x,y) = g^{-1}(g(x)+g(y)) - x-y$ by \eqref{Eq:PairDef}, and $g(x+y)=g(x)+g(y)$ follows. Suppose that the claim holds for some $i$. By (iii), $g^k(x)\in\rad{\gamma}$ for every integer $k$. By the induction assumption, we then have $g^{i+1}(x+y) = g^i(g(x+y)) = g^i(g(x)+g(y)) = g^i(g(x))+g^i(g(y)) = g^{i+1}(x)+g^{i+1}(y)$. Furthermore, $g^i(g^{-i}(x)+g^{-i}(y)) = g^i(g^{-i}(x))+g^i(g^{-i}(y)) = x+y$ and hence $g^{-i}(x+y)=g^{-i}(x)+g^{-i}(y)$. The claim therefore holds for every $i$.

Part (v) follows from (iii) and (iv).
\end{proof}

\begin{lemma}\label{Lm:iPairProp}
Let $(g,\gamma)$ be a construction pair on $(X,+)$. Then
\begin{equation}\label{Eq:iPairProp}
    \gamma(g^ix,g^jy) = g^{-i-j}\gamma(x,y)
\end{equation}
for all $x,y\in X$ and $i,j\in\mathbb Z$.
\end{lemma}
\begin{proof}
It suffices to show
\begin{equation}\label{Eq:iPPAux}
    \gamma(g^ix,y) = g^{-i}\gamma(x,y)
\end{equation}
for all $i\in\mathbb Z$, since then $\gamma(g^ix,g^jy) = g^{-i}\gamma(x,g^jy) = g^{-i}\gamma(g^jy,x) = g^{-i}g^{-j}\gamma(y,x) = g^{-i-j}\gamma(x,y)$ by symmetry of $\gamma$. Let us prove \eqref{Eq:iPPAux}. For $i=0$ there is nothing to show and for $i=1$ we recover \eqref{Eq:PairProp}. If \eqref{Eq:iPPAux} holds for $i$ then $\gamma(g^{i+1}x,y) = \gamma(g^igx,y) = g^{-i}\gamma(gx,y) = g^{-i}g^{-1}\gamma(x,y) = g^{-i-1}\gamma(x,y)$ shows that \eqref{Eq:iPPAux} holds for $i+1$. To prove \eqref{Eq:iPPAux} for $i=-1$, substitute $g^{-1}(x)$ for $x$ in \eqref{Eq:PairProp} to obtain $g^{-1}\gamma(g^{-1}x,y) = \gamma(x,y)$, so $\gamma(g^{-1}x,y) = g\gamma(x,y)$. Finally, if \eqref{Eq:iPPAux} holds for $i$ then $\gamma(g^{i-1}x,y) = \gamma(g^ig^{-1}x,y) = g^{-i}\gamma(g^{-1}x,y) = g^{-i}g\gamma(x,y) = g^{-i+1}\gamma(x,y)$ shows that \eqref{Eq:iPPAux} holds for $i-1$.
\end{proof}

\begin{lemma}\label{Lm:PairInverse}
Let $(g,\gamma)$ be a construction pair on $(X,+)$. Then $(g^{-1},g\gamma)$ is a construction pair on $(X,+)$.
\end{lemma}
\begin{proof}
We have $g\gamma(x,y) = g\gamma(y,x)$ by symmetry of $\gamma$. As $g(0)=0$ by Lemma \ref{Lm:PairRadImg}, we have $g\gamma(x,x) = g(0)=0$. Since $g$ restricts to an automorphism of $\rad{\gamma}$ and $\img{\gamma}\subseteq\rad{\gamma}$ by Lemma \ref{Lm:PairRadImg}, we have $g\gamma(x+y,z) = g(\gamma(x,z)+\gamma(y,z)) = g\gamma(x,z)+g\gamma(y,z)$. This proves that that $g\gamma$ is symmetric, alternating and biadditive.

Now, $g\gamma(g\gamma(x,y),z) = gg^{-1}\gamma(\gamma(x,y),z) = \gamma(\gamma(x,y),z) = 0$ by \eqref{Eq:PairProp} and \eqref{Eq:PairRad}. Hence $\img{g\gamma}\subseteq\rad{g\gamma}$.

Finally, $g(x)+g(y) = g(x+y+\gamma(x,y)+g^{-1}\gamma(x,y)+g^{-2}\gamma(x,y)) = g(x+y)+g\gamma(x,y)+\gamma(x,y)+g^{-1}\gamma(x,y)$ by Lemma \ref{Lm:PairRadImg}. Substituting $g^{-1}(x)$ for $x$, $g^{-1}(y)$ for $y$ and using \eqref{Eq:iPairProp} yields $x+y = g(g^{-1}x+g^{-1}y) + g^3\gamma(x,y)+g^2\gamma(x,y)+g\gamma(x,y)$. Since $\gamma(x,y)=-\gamma(x,y)$ and $g(-z)=-g(z)$ by Lemma \ref{Lm:PairRadImg}, we deduce
\begin{displaymath}
    g(g^{-1}x+g^{-1}y) = x+y+g\gamma(x,y)+g^2\gamma(x,y)+g^3\gamma(x,y),
\end{displaymath}
which is the analog of \eqref{Eq:PairDef} for $(g^{-1},g\gamma)$.
\end{proof}

Lemma \ref{Lm:Interval} lists basic properties of the intervals $I(i,j)$ of \eqref{Eq:Interval}. For subsets $U$, $V\subseteq\mathbb Z$, let
\begin{displaymath}
    U\oplus V = (U\cup V)\setminus(U\cap V)
\end{displaymath}
be the symmetric difference of $U$ an $V$.

\begin{lemma}\label{Lm:Interval}
Let $i,j,k$ be integers. Then:
\begin{enumerate}
\item[(i)] $I(i,j)=I(j,i)$,
\item[(ii)] $I(i,j)+k = I(i+k,j+k)$,
\item[(ii)] $I(i,k)\oplus I(j,k) = I(i,j)$.
\end{enumerate}
\end{lemma}
\begin{proof}
Parts (i) and (ii) are clear from the definition of $I(i,j)$. For (iii), thanks to symmetry it suffices to discuss the three cases $i\le j\le k$, $i\le k\le j$ and $k\le i\le j$. Note that then $I(i,j) = [i,j) = \{i,i+1,\dots,j-1\}$. If $i\le j\le k$ then $I(i,k)\oplus I(j,k) = [i,k)\oplus [j,k) = [i,j)$, if $i\le k\le j$ then $I(i,k)\oplus I(j,k) = [i,k)\oplus [k,j) = [i,j)$, and if $k\le i\le j$ then $I(i,k)\oplus I(j,k) = [k,i)\oplus [k,j) = [i,j)$.
\end{proof}

\begin{lemma}
Let $(g,\gamma)$ be a construction pair on $(X,+)$. Then
\begin{equation}\label{Eq:iPairDef}
    g^{-i}(g^ix+g^iy) = x+y + \lcsum{k\in I(0,3i)}g^{-k}\gamma(x,y)
\end{equation}
for all $x,y\in X$ and $i,j\in\mathbb Z$.
\end{lemma}
\begin{proof}
There is nothing to prove for $i=0$ since $I(0,0)=\emptyset$. For $i=1$ we get $I(0,3i)=I(0,3)=\{0,1,2\}$ and we recover \eqref{Eq:PairDef}. If \eqref{Eq:iPairDef} holds for some $i\ge 0$, then
\begin{displaymath}
    g^{-i-1}(g^{i+1}x+g^{i+1}y) = g^{-1}g^{-i}(g^igx+g^ig(y)) = g^{-1}\Big(gx+gy+\lcsum{k\in I(0,3i)}g^{-k}\gamma(gx,gy)\Big).
\end{displaymath}
We have $\gamma(gx,gy) = g^{-2}\gamma(x,y)$ by \eqref{Eq:iPairProp}. By Lemma \ref{Lm:PairRadImg}, $g^{-k}\gamma(gx,gy)\in\img{\gamma}\subseteq\rad{\gamma}$, and a power of $g$ behaves as a homomorphism if at least one of the summands lies in $\rad{\gamma}$. Therefore $g^{-i-1}(g^{i+1}x+g^{i+1}y)$ is equal to
\begin{multline*}
	g^{-1}(gx+gy) + \lcsum{k\in I(0,3i)}g^{-1}g^{-k}g^{-2}\gamma(x,y) = g^{-1}(gx+gy) + \lcsum{k\in I(0,3i)}g^{-k-3}\gamma(x,y)\\
      = x+y+\sum_{k\in\{0,1,2\}}g^{-k}\gamma(x,y) + (g^{-3}\gamma(x,y)+\cdots+g^{-3i-2}\gamma(x,y)) = x+y+\lcsum{k\in I(0,3(i+1))}g^{-k}\gamma(x,y),
\end{multline*}
which gives \eqref{Eq:iPairDef} for $i+1$.

By Lemma \ref{Lm:PairInverse}, $(g^{-1},g\gamma)$ is a construction pair on $(X,+)$. For $i=-1$ we get $I(0,3i) = I(0,-3) = \{-3,-2,-1\}$ and \eqref{Eq:iPairDef} becomes
\begin{displaymath}
    g(g^{-1}x+g^{-1}y) = x+y + g\gamma(x,y)+g^2\gamma(x,y)+g^3\gamma(x,y),
\end{displaymath}
which is the analog of \eqref{Eq:PairDef} for $(g^{-1},g\gamma)$. If \eqref{Eq:iPairDef} holds for some $i\le 0$, then
\begin{displaymath}
    g^{-i+1}(g^{i-1}x+g^{i-1}y) = gg^{-i}(g^ig^{-1}x+g^ig^{-1}y) = g\Big(g^{-1}x+g^{-1}y+\lcsum{k\in I(0,3i)}g^{-k}\gamma(g^{-1}x,g^{-1}y)\Big).
\end{displaymath}
We have $\gamma(g^{-1}x,g^{-1}y) = g^2\gamma(x,y)$ by \eqref{Eq:iPairProp}. Using Lemma \ref{Lm:PairRadImg} as above, we see that $g^{-i+1}(g^{i-1}x+g^{i-1}y)$ is equal to
\begin{multline*}
	g(g^{-1}x+g^{-1}y) + \lcsum{k\in I(0,3i)}gg^{-k}g^2\gamma(x,y) = g(g^{-1}x+g^{-1}y) + \lcsum{k\in I(0,3i)}g^{-k+3}\gamma(x,y)\\
        = x+y+\sum_{k\in\{1,2,3\}} g^k\gamma(x,y)+(g^4\gamma(x,y)+\cdots+g^{3i+3}\gamma(x,y)) = x+y+\lcsum{k\in I(0,3(i-1))}g^{-k}\gamma(x,y),
\end{multline*}
which gives \eqref{Eq:iPairDef} for $i-1$.
\end{proof}

\begin{corollary}\label{Cr:iPairFormula}
Let $(g,\gamma)$ be a construction pair on $(X,+)$ and let $i\in\mathbb Z$. Then
\begin{equation}\label{Eq:iPairFormula}
    g^{-i}(x+y) = g^{-i}(x)+g^{-i}(y) + \lcsum{k\in I(-2i,i)}g^{-k}\gamma(x,y)
\end{equation}
for every $x,y\in X$.
\end{corollary}
\begin{proof}
By \eqref{Eq:iPairDef} we have
\begin{displaymath}
    g^{-i}(x+y) = g^{-i}(g^ig^{-i}x+g^ig^{-i}y) = g^{-i}(x)+g^{-i}(y) + \lcsum{k\in I(0,3i)}g^{-k}\gamma(g^{-i}x,g^{-i}y).
\end{displaymath}
By Lemma \ref{Lm:iPairProp}, $\gamma(g^{-i}x,g^{-i}y) = g^{2i}\gamma(x,y)$. Therefore
\begin{displaymath}
    \lcsum{k\in I(0,3i)}g^{-k}\gamma(g^{-i}x,g^{-i}y) = \lcsum{k\in I(0,3i)}g^{-k+2i}\gamma(x,y)  = \lcsum{k\in I(0,3i)-2i}g^{-k}\gamma(x,y).
\end{displaymath}
As $I(0,3i)-2i = I(-2i,i)$ by Lemma \ref{Lm:Interval}, we are done.
\end{proof}

\section{Moufang loops obtained from construction pairs}\label{Sc:MCP}

In this section we introduce a multiplication formula on $C\times X$ for a cyclic group $C$ and a construction pair $(g,\gamma)$ on $(X,+)$, cf.~\eqref{Eq:MultGeneral}. We give a necessary and sufficient condition for the resulting loop to be a Moufang loop.

To further simplify the notation employed in Section \ref{Sc:CPP}, let
\begin{displaymath}
    \mysum Uxy = \sum_{k\in U}g^{-k}\gamma(x,y),
\end{displaymath}
for a fixed construction pair $(g,\gamma)$ on $(X,+)$ and any $U\subseteq\mathbb Z$ and $x,y\in\mathbb Z$.

\begin{lemma}\label{Lm:MySum}
Let $(g,\gamma)$ be a construction pair on $(X,+)$, $U,V\subseteq\mathbb Z$, $x,y\in X$ and $i,j\in\mathbb Z$. Then:
\begin{enumerate}
\item[(i)] $\mysum U{g^ix}{g^jy} = \mysum {U+i+j}xy$,
\item[(ii)] $g^i\mysum Uxy = \mysum{U-i}xy$,
\item[(iii)] $\mysum Uxy+\mysum Vxy = \mysum{U\oplus V}xy$
\end{enumerate}
\end{lemma}
\begin{proof}
For (i), we can use by \eqref{Eq:iPairProp} to get
\begin{displaymath}
    \mysum U{g^ix}{g^jy} = \lcsum{k\in U}g^{-k}\gamma(g^ix,g^jy) = \lcsum{k\in U}g^{-k-i-j}\gamma(x,y)
   	= \lcsum{k\in U+i+j}g^{-k}\gamma(x,y) = \mysum{U+i+j}xy.
\end{displaymath}
By Lemma \ref{Lm:PairRadImg}, $g$ restricts to an automorphism of $\rad{\gamma}$ and $\img{\gamma}\subseteq\rad{\gamma}$. Hence
\begin{displaymath}
	g^i\mysum Uxy = g^i\sum_{k\in U}g^{-k}\gamma(x,y) = \sum_{k\in U}g^ig^{-k}\gamma(x,y) = \sum_{k\in U}g^{-(k-i)}\gamma(x,y) = \mysum{U-i}xy,
\end{displaymath}
proving (ii). Part (iii) follows immediately from $2\img{\gamma}=0$ of Lemma \ref{Lm:PairRadImg}.
\end{proof}

For a construction pair $(g,\gamma)$ on $(X,+)$, let $r(g,\gamma)$ be the least positive integer $r$ such that
\begin{equation}\label{Eq:r}
	\lcsum{0\le k<r}g^k(x)\in\rad{\gamma}
\end{equation}
for every $x\in X$, if such $r$ exists. Otherwise set $r(g,\gamma)=\infty$.

\begin{lemma}\label{Lm:r1}
Let $(g,\gamma)$ be a construction pair on $(X,+)$. The following conditions are equivalent for an integer $n$ and $x\in X$:
\begin{enumerate}
\item[(i)] $\sum_{0\le k<|n|}g^k(x)\in\rad{\gamma}$,
\item[(ii)] $\mysum{I(0,n)}xy=0$ for all $y\in X$,
\item[(iii)] $\mysum{I(i,i+n)}xy=0$ for all $y\in X$ and $i\in\mathbb Z$.
\end{enumerate}
\end{lemma}
\begin{proof}
If $n=0$ then all three sums vacuously evaluate to $0$. Suppose that $n\ne 0$. We have $g^{-i}\mysum{I(0,n)}xy = \mysum{I(0,n)+i}xy = \mysum{I(i,i+n)}xy$ by Lemma \ref{Lm:MySum}. Since $g(0)=0$, the conditions (ii) and (iii) are equivalent. Let us establish the equivalence of (i) and (ii). Note that condition (i) is unchanged when $n$ is replaced by $-n$. Since $g^n\mysum{I(0,n)}xy = \mysum{I(-n,0)}xy = \mysum{I(0,-n)}xy$ and $g(0)=0$, condition (ii) is also unchanged when $n$ is replaced by $-n$. We can thus assume that $n>0$. Now, $\mysum{I(0,n)}xy = \sum_{0\le k<n}g^{-k}\gamma(x,y) = \sum_{0\le k<n}\gamma(g^kx,y) = \gamma(\sum_{0\le k<n}g^kx,y)$ by Lemma \ref{Lm:iPairProp} and biadditivity of $\gamma$. Hence (i) and (ii) are equivalent.
\end{proof}

\begin{lemma}\label{Lm:r2}
Let $(g,\gamma)$ be a construction pair on $(X,+)$ and let $r=r(g,\gamma)$. The following conditions are equivalent for an integer $n$:
\begin{enumerate}
\item[(i)] $n=0$ or $r$ divides $n$,
\item[(ii)] $\sum_{0\le k<|n|}g^k(x)\in\rad{\gamma}$ for all $x\in X$,
\item[(iii)] $\mysum{I(0,n)}xy=0$ for all $x,y\in X$,
\item[(iv)] $\mysum{I(i,i+n)}xy=0$ for all $x,y\in X$ and $i\in\mathbb Z$.
\end{enumerate}
\end{lemma}
\begin{proof}
If $n=0$ then all three conditions are clearly equivalent, so suppose that $n\ne 0$. The equivalence of (ii), (iii) and (iv) follows from Lemma \ref{Lm:r1}. Let us prove the equivalence of (i) and (iii). As in the proof of Lemma \ref{Lm:r1}, we can suppose without loss of generality that $n>0$. Note that if $r$ is finite then Lemma \ref{Lm:r1} implies $\mysum{I(i,i+r)}xy=0$ for every $i$.

Suppose that (i) holds. Then $r<\infty$ and $n=qr$ for some positive integer $q$. Since $I(0,n) = I(0,r)\cup I(r,2r)\cup\cdots\cup I((q-1)r,n)$, we have $\mysum{I(0,n)}xy= \mysum{I(0,r)}xy+\mysum{I(r,2r)}xy + \cdots + \mysum{I((q-1)r,n)}xy = 0$. Conversely, suppose that (ii) holds, which again implies that $r<\infty$. Suppose for a contradiction that $n=qr+t$ for some $0<t<r$. Then $\mysum{I(0,t)}xy = \mysum{I(0,n)}xy - \mysum{I(t,n)}xy = \mysum{I(0,n)}xy - g^{-t}\mysum{I(0,qr)}xy = 0-g^{-t}(0)=0$ for all $x,y\in X$. Then $\sum_{0\le k<t}g^kx\in\rad{\gamma}$ for all $x\in X$ by Lemma \ref{Lm:r1}, a contradiction with the minimality of $r$.
\end{proof}

\begin{lemma}\label{Lm:WellDefined}
Let $(X,+)$ be an abelian group, $(g,\gamma)$ a construction pair on $(X,+)$ and $C=\langle b\rangle$ a cyclic group. Then the following conditions are equivalent:
\begin{enumerate}
\item[(i)] The formula
\begin{equation}\label{Eq:MultGeneral}
    (b^i,x)\cdot(b^j,y) = \Big(b^{i+j},g^{-j}(x)+y+\lcsum{k\in I(i+j,-j)}g^{-k}\gamma(x,y)\Big)
\end{equation}
correctly defines a multiplication on $C\times X$.
\item[(ii)] Either $C$ is infinite, or $C$ is finite, $|g|$ divides $|C|$ and $\sum_{0\le k<|C|}g^k(x)\in\rad{\gamma}$ for every $x\in X$.
\item[(iii)] Either $C$ is infinite, or $C$ is finite, $|g|$ divides $|C|$ and $r(g,\gamma)$ divides $|C|$.
\end{enumerate}
\end{lemma}
\begin{proof}
The equivalence of (ii) and (iii) follows from Lemma \ref{Lm:r2}. Let us prove the equivalence of (i) and (ii).

Suppose that the multiplication \eqref{Eq:MultGeneral} is well defined. If $C$ is infinite, we are done, so suppose that $|C|=n$ is finite. Recall that $\gamma(x,0)=0$ and $g(0)=0$ by Lemma \ref{Lm:PairRadImg}. Then for every $x\in X$ the product $(b^0,x)\cdot (b^0,0)$ evaluates to $(b^0,x)$ and is equal to $(b^n,x)\cdot (b^0,0) = (b^0,g^{-n}(x))$, so $g^n(x)=x$ and $|g|$ divides $n$. Also, for every $x,y\in X$ the product $(b^0,x)\cdot (b^0,y)=(b^0,x+y)$ is equal to $(b^n,x)\cdot (b^0,y) = (b^0,x+y+\mysum{I(n,0)}xy)$, so $\mysum{I(0,n)}xy=0$ and hence $\sum_{0\le k<n}g^k(x)\in\rad{\gamma}$ by Lemma \ref{Lm:r2}.

Conversely, if $C$ is infinite then certainly the multiplication is well-defined. Suppose that $|C|=n$, $|g|$ divides $n$ and $\sum_{0\le k<n}g^k(x)\in\rad{\gamma}$ for every $x\in X$. All we need to check is that the formula \eqref{Eq:MultGeneral} gives the same result when $i$ is replaced by $i+n$ and when $j$ is replaced by $j+n$. Since we are assuming that $|g|$ divides $n$, we only need to check
\begin{displaymath}
    \mysum{I(i+n+j,-j)}xy\ =\ \mysum{I(i+j,-j)}xy\ =\ \mysum{I(i+j+n,-j-n)}xy. 
\end{displaymath}
Recall that signs play no role here since $2\img{\gamma}=0$. For the first equality, note that $I(i+n+j,-j)\oplus I(i+j,-j) = I(i+j+n,i+j)$ by Lemma \ref{Lm:Interval}, so we need to check that $\mysum{I(i+j+n,i+j)}xy = 0$. This is true by Lemma \ref{Lm:r2} since $\sum_{0\le k<n}g^k(x)\in\rad{\gamma}$ for every $x\in X$. For the second equality, it now suffices to show $\mysum{I(i+n+j,-j)}xy  = \mysum{I(i+j+n,-j-n)}xy$, which is the same as $\mysum{I(-j,-j-n)}xy=0$ by Lemma \ref{Lm:Interval}, which again holds by Lemma \ref{Lm:r2}.
\end{proof}

The following result shows (in combination with Lemma \ref{Lm:r2}) that the necessary condition $\sum_{0\le k<|C|}g^k(x)\in\rad{\gamma}$ of Lemma \ref{Lm:WellDefined} is automatically satisfied when $|g|$ divides $|C|$ and $\rad{\gamma}$ contains no elements of order $3$.

\begin{lemma}\label{Lm:r3}
Let $(g,\gamma)$ be a construction pair on $(X,+)$ and let $r=r(g,\gamma)$. If $|g|<\infty$ then $r$ divides $3|g|$. If $|g|<\infty$ and $\rad{\gamma}$ contains no elements of order $3$ then $r$ divides $|g|$.
\end{lemma}
\begin{proof}
Let $m=|g|<\infty$. By \eqref{Eq:iPairDef}, $x+y = g^{-m}(g^mx+g^my) = x+y + \mysum{I(0,3m)}xy$, so $\mysum{I(0,3m)}xy=0$ and $r$ divides $3m$ by Lemma \ref{Lm:r2}. Suppose also that $\rad{\gamma}$ contains no elements of order $3$. Since $g^k = g^{k+m}$ for every $k$, we have $0 = \mysum{I(0,3m)}xy = 3\mysum{I(0,m)}xy$. The sum $\mysum{I(0,m)}xy$ lies in $\rad{\gamma}$ by Lemma \ref{Lm:PairRadImg}. Since $\rad{\gamma}$ contains no elements of order $3$, we deduce $\mysum{I(0,m)}xy=0$ and then Lemma \ref{Lm:r2} implies that $r$ divides $m$.
\end{proof}

Here is the main construction of the paper:

\begin{definition}\label{Df:Ext}
Suppose that  $(X,+)$ is an abelian group, $(g,\gamma)$ is a construction pair on $(X,+)$ and $C=\langle b\rangle$ is a cyclic group. If $C$ is finite, suppose further that $|g|$ divides $|C|$ and $r(g,\gamma)$ divides $|C|$. Then the groupoid $(C\times X,\cdot)$ with multiplication \eqref{Eq:MultGeneral} will be denoted by $C\ltimes_{(g,\gamma)}X$ or $\langle b\rangle\ltimes_{(g,\gamma)}X$. Whenever we use this notation, we implicitly assume that all conditions required by this definition are satisfied.
\end{definition}

Note that $C\times 0$ is a subloop of $Q = C\ltimes_{(g,\gamma)}X$. The mapping $(b^i,x)\mapsto b^i$ is a homomorphism from $Q = C\ltimes_{(g,\gamma)}X$ onto $C$ with kernel $1\times X$. Hence $1\times X$ is a normal subloop of $Q$ such that $Q/(1\times X)$ is isomorphic to $C$. Obviously, $(C\times 0)\cap(1\times X)=1$, so $Q$ is split.

\begin{theorem}\label{Th:AbelianByCyclic}
Let $Q=C\ltimes_{(g,\gamma)}X$ be the groupoid from Definition $\ref{Df:Ext}$. Then $Q$ is a Moufang loop with neutral element $(1,0)$ and inverses $(b^i,x)^{-1} = (b^{-i},-g^i(x))$.
\end{theorem}
\begin{proof}
We will use Lemmas \ref{Lm:PairRadImg}, \ref{Lm:Interval} and \ref{Lm:MySum} freely. By Lemma \ref{Lm:WellDefined}, the multiplication is well defined.

\textit{Neutral element.} Substituting $(1,0)=(b^0,0)$ for $(b^i,x)$ into \eqref{Eq:MultGeneral} yields $(b^0,0)\cdot(b^j,y) = (b^j,g^{-j}(0)+y+\mysum{I(j,-j)}0y)$. Since $\gamma(0,y)=0=g(0)$, the product reduces to $(b^j,y)$. Similarly, $(b^i,x)\cdot (b^0,0) = (b^i,g^{-0}(x)+0+\mysum{I(i,0)}x0) = (b^i,x)$.

\textit{Two-sided inverses.} We will show that $(b^{-i},-g^i(x))$ is the two-sided inverse of $(b^i,x)$. We have $(b^i,x)\cdot(b^{-i},-g^i(x)) = (1,g^i(x)-g^i(x)+\mysum{I(0,i)}{x}{-g^i(x)})$. Since $\gamma(x,-g^i(x)) = -g^{-i}\gamma(x,x)=-g^{-i}(0)=0$, the product reduces to $(1,0)$. Similarly, $(b^{-i},-g^i(x))\cdot (b^i,x) = (1,g^{-i}(-g^i(x))+x+\mysum{I(0,-i)}{-g^i(x)}{x}) = (1,0)$ because the sum again vanishes and $g^k(-x)=-g^k(x)$.

\textit{Inverse property loop.} It is well known, cf.~\cite{PhillipsVojtechovsky}, that a groupoid with a neutral element and two-sided inverses is an inverse property loop if $x^{-1}(xy)=y=(yx)x^{-1}$ holds for all $x,y$. We will verify the two identities. Let us write $b^ix$ instead of $(b^i,x)$ from now on. We have
\begin{align*}
    b^{-i}(-g^ix)\cdot (b^ix\cdot b^jy)& =
        b^{-i}(-g^ix)\cdot b^{i+j}(g^{-j}x+y+\mysum{I(i+j,-j)}xy)\\
        &=b^j(g^{-i-j}(-g^ix) + g^{-j}x+y+\mysum{I(i+j,-j)}xy + \mysum{I(j,-i-j)}{-g^ix}{g^{-j}x+y}).
\end{align*}
Note that $\gamma(g^ax,g^bx+y) = \gamma(g^ax,g^bx)+\gamma(g^ax,y) = g^{-a-b}\gamma(x,x) +\gamma(g^ax,y) = \gamma(g^ax,y)$ since $\gamma$ is biadditive and alternating. Hence
\begin{displaymath}
    \mysum{I(j,-i-j)}{-g^ix}{g^{-j}x+y} = \mysum{I(j,-i-j)}{-g^ix}{y} = -\mysum{I(j,-i-j)}{g^ix}{y} = -\mysum{I(i+j,-j)}xy.
\end{displaymath}
Also, $g^{-i-j}(-g^ix) = -g^{-i-j}g^i(x) = -g^{-j}x$. Combining, we get $b^{-i}(-g^ix)\cdot (b^ix\cdot b^jy)=b^jy$. Similarly,
\begin{align*}
    (b^ix\cdot b^jy)\cdot b^{-j}(-g^jy)
		&= b^{i+j}(g^{-j}x+y+\mysum{I(i+j,-j)}xy)\cdot b^{-j}(-g^jy)\\
        	&= b^i(g^j(g^{-j}(x)+y+\mysum{I(i+j,-j)}xy)-g^jy+\mysum{I(i,j)}{g^{-j}x+y}{-g^jy}),
\end{align*}
where we took advantage of $\mysum{I(i+j,-j)}xy\in\img{\gamma}\subseteq\rad{\gamma}$. By Corollary \ref{Cr:iPairFormula},
\begin{align*}
    	g^j(g^{-j}(x)+y+\mysum{I(i+j,-j)}xy)
    	&=g^j(g^{-j}x)+g^j(y+\mysum{I(i+j,-j)}xy) + \mysum{I(2j,-j)}{g^{-j}x}{y}\\
	&=x+g^jy+ g^j\mysum{I(i+j,-j)}xy + \mysum{I(2j,-j)}{g^{-j}x}{y}\\
	&=x+g^jy+\mysum{I(i,-2j)}xy+\mysum{I(j,-2j)}xy = x+g^jy+\mysum{I(i,j)}xy.
\end{align*}
Also,
\begin{displaymath}
    \mysum{I(i,j)}{g^{-j}x+y}{-g^jy} = -\mysum{I(i,j)}{g^{-j}x}{g^jy} = -\mysum{I(i,j)}xy.
\end{displaymath}
Combining, we get $(b^ix\cdot b^jy)\cdot b^{-j}(-g^jy) = b^ix$.

\textit{Moufang identity.} It remains to show that $(C\times X,\cdot)$ satisfies one of the Moufang identities. We will establish \eqref{Eq:M2}. First we find a formula for $b^ix\cdot (b^jy\cdot b^ix)$. We have
\begin{multline*}
    b^ix\cdot (b^jy\cdot b^ix) = b^ix\cdot b^{i+j}(g^{-i}y+x+\mysum{I(i+j,-i)}xy\\
    =b^{2i+j}(g^{-i-j}x+g^{-i}y+x+\mysum{I(i+j,-i)}xy+\mysum{I(2i+j,-i-j)}{x}{g^{-i}y+x}).
\end{multline*}
Since
\begin{displaymath}
    \mysum{I(2i+j,-i-j)}{x}{g^{-i}y+x}\ =\ \mysum{I(2i+j,-i-j)}{x}{g^{-i}y}\ =\ \mysum{I(i+j,-2i-j)}xy 
\end{displaymath}
and $I(i+j,-2i-j)\oplus I(i+j,-i) = I(-i,-2i-j)$, we get
\begin{equation}\label{Eq:Flex}
    b^ix\cdot (b^jy\cdot b^ix) = b^{2i+j}(g^{-i-j}x+x+g^{-i}y+\mysum{I(-i,-2i-j)}xy.
\end{equation}
Using \eqref{Eq:Flex} in the second equality below, we have
\begin{multline*}
    b^ix\cdot ((b^jy\cdot b^kz)\cdot b^ix) =b^ix\cdot(b^{j+k}(g^{-k}y+z+\mysum{I(j+k,-k)}yz)\cdot b^ix)\\
    =b^{2i+j+k}(g^{-i-j-k}x+x+g^{-i}(g^{-k}y+z)+g^{-i}\mysum{I(j+k,-k)}yz + \mysum{I(-i,-2i-j-k)}{x}{g^{-k}y+z}).
\end{multline*}
We will next rewrite certain three summands of the last equation, using Corollary \ref{Cr:iPairFormula} for the first summand. We get:
\begin{align*} 
    &g^{-i}(g^{-k}y+z) = g^{-i-k}y+g^{-i}z+\mysum{I(-2i,i)}{g^{-k}y}{z} = g^{-i-k}y+g^{-i}z+\mysum{I(-2i-k,i-k)}yz,\\
    &g^{-i}\mysum{I(j+k,-k)}yz\ =\  \mysum{I(i+j+k,i-k)}yz,\\
    &\mysum{I(-i,-2i-j-k)}{x}{g^{-k}y+z}\ \ =\ \ \mysum{I(-i-k,-2i-j-2k)}xy\  \ +\ \ \mysum{I(-i,-2i-j-k)}xz.
\end{align*}
Combining the two sums for $(y,z)$, we see that $b^ix\cdot ((b^jy\cdot b^kz)\cdot b^ix) = b^{2i+j+k}u$, where $u$ is equal to
\begin{displaymath}
    g^{-i-j-k}x+x+g^{-i-k}y+g^{-i}z+\ \ \mysum{I(-i-k,-2i-j-2k)}xy\ \ +\ \ \mysum{I(-i,-2i-j-k)}xz\  \ +\ \ \mysum{I(i+j+k,-2i-k)}yz.
\end{displaymath}
On the other hand,
\begin{displaymath}
    (b^ix\cdot b^jy)\cdot (b^kz\cdot b^ix)
    =b^{i+j}(g^{-j}x+y+\mysum{I(i+j,-j)}xy)\cdot b^{i+k}(g^{-i}z+x+\mysum{I(i+k,-i)}xz)
\end{displaymath}
is equal to $b^{2i+j+k}v$, where $v$ is
\begin{displaymath} 
    g^{-i-k}(g^{-j}x+y)+g^{-i-k}\mysum{I(i+j,-j)}xy + g^{-i}z+x+\ \mysum{I(i+k,-i)}xz \ +\  \mysum{I(2i+j+k,-i-k)}{g^{-j}x+y}{g^{-i}z+x}.
\end{displaymath}
Focusing on certain three summands again, we get
\begin{align*} 
    	&g^{-i-k}(g^{-j}x+y) = g^{-i-j-k}x+g^{-i-k}y+\mysum{I(-2i-2k,i+k)}{g^{-j}x}{y} = g^{-i-j-k}x+g^{-i-k}y+\mysum{I(-2i-j-2k,i-j+k)}xy,\\
	&g^{-i-k}\mysum{I(i+j,-j)}xy = \mysum{I(2i+j+k,i-j+k)}xy,\\
   	&\ \mysum{I(2i+j+k,-i-k)}{g^{-j}x+y}{g^{-i}z+x} =\ \ \mysum{I(2i+j+k,-i-k)}xy\ \ +\ \ \mysum{I(i+k,-2i-j-k)}xz\ \ +\ \ \mysum{I(i+j+k,-2i-k)}yz.
\end{align*}
The overall contribution of $(x,y)$ to $(b^ix\cdot b^jy)\cdot (b^kz\cdot b^ix)$ is therefore
\begin{multline*}
	I(-2i-j-2k,i-j+k) \oplus I(2i+j+k,i-j+k)\oplus I(2i+j+k,-i-k) \\
	= I(-2i-j-2k,2i+j+k)\oplus I(2i+j+k,-i-k) = I(-i-k,-2i-j-2k),
\end{multline*}
while the overall contribution of $(x,z)$ is
\begin{displaymath}
    I(i+k,-i)\oplus I(i+k,-2i-j-k) = I(-i,-2i-j-k).
\end{displaymath}
Combining, we see that $v=u$ and \eqref{Eq:M2} holds.
\end{proof}

\section{Properties of the Moufang loops obtained from construction pairs}\label{Sc:PropertiesCP}

In this section we establish some structural properties of the Moufang loops $C\ltimes_{(g,\gamma)}X$ of Theorem \ref{Th:AbelianByCyclic}. We will freely use Lemmas \ref{Lm:PairRadImg}, \ref{Lm:Interval} and \ref{Lm:MySum}.

\begin{lemma}\label{Lm:Commutator}
Let $Q=\langle b\rangle\ltimes_{(g,\gamma)}X$. For all $i,j\in\mathbb Z$ and $x,y\in X$ we have
\begin{equation}\label{Eq:Commutator}
	[(b^i,x),(b^j,y)] = \Big(1,-x+g^{-j}x+y-g^{-i}y+\mysum{I(-i,-j)}xy + \mysum{I(-i-j,i+j)}xy\Big).
\end{equation}
\end{lemma}
\begin{proof}
Let $(1,u)$ be the right hand side of \eqref{Eq:Commutator}. We need to check that $(b^j,y)(b^i,x)\cdot (1,u) = (b^i,x)(b^j,y)$. Let
\begin{align*}
    v &= \mysum{I(i+j,0)}{g^{-i}y+x}{-x+g^{-j}x+y-g^{-i}y}\\
	& = \mysum{I(i+j,0)}{g^{-i}y}{x} + \mysum{I(i+j,0)}{g^{-i}y}{g^{-j}x} + \mysum{I(i+j,0)}xy + \mysum{I(i+j,0)}{x}{g^{-i}y}\\
      &= \mysum{I(i+j,0)}{g^{-i}y}{g^{-j}x} + \mysum{I(i+j,0)}xy = \mysum{I(0,-i-j)}xy+\mysum{I(i+j,0)}xy = \mysum{I(-i-j,i+j)}xy.
\end{align*}
Then
\begin{align*}
	&(b^j,y)(b^i,x)\cdot (1,u) = (b^{i+j},g^{-i}y+x+\mysum{I(i+j,-i)}xy)\cdot (1,u)\\
	&=(b^{i+j},g^{-i}y+x+\mysum{I(i+j,-i)}xy -x+g^{-j}x+y-g^{-i}y+\mysum{I(-i,-j)}xy+\mysum{I(-i-j,i+j)}xy + v)\\
	&=(b^{i+1},g^{-j}x+y+\mysum{I(i+j,-i)}xy+\mysum{I(-i,-j)}xy+\mysum{I(-i-j,i+j)}xy+\mysum{I(-i-j,i+j)}xy)\\
	&=(b^{i+j},g^{-j}x+y+\mysum{I(i+j,-j)}xy) = (b^i,x)(b^j,y).\qedhere
\end{align*}
\end{proof}

\begin{lemma}\label{Lm:Associator}
Let $Q=\langle b\rangle\ltimes_{(g,\gamma)}X$. For all $i,j,k\in\mathbb Z$ and $x,y,z\in X$ we have
\begin{equation}\label{Eq:Associator} 
	[(b^i,x),(b^j,y),(b^k,z)] = \Big(1,\ \ \ \ \mysum{I(i+j,i+j+k)}xy\ \ +\ \ \mysum{I(i+k,i+j+k)}xz\ \ +\ \ \mysum{I(j+k,i+j+k)}yz\Big).
\end{equation}
\end{lemma}
\begin{proof}
Let $(1,u)$ be the right hand side of the associator formula \eqref{Eq:Associator}. Our goal is to check the equation $((b^i,x)\cdot (b^j,y)(b^k,z))(1,u) = (b^i,x)(b^j,y)\cdot (b^k,z)$. Direct computation yields
\begin{displaymath} 
	(b^i,x)\cdot (b^j,y)(b^k,z)
	= (b^{i+j+k},g^{-j-k}x+g^{-k}y+z+\ \mysum{I(i+j,-j-2k)}xy\ +\ \mysum{I(i+j+k,-j-k)}xz\ +\ \mysum{I(j+k,-k)}yz).
\end{displaymath}
Upon noting that  $u\in\rad{\gamma}$, we see that $((b^i,x)\cdot (b^j,y)(b^k,z))(1,u)$ is therefore equal to
\begin{displaymath} 
	(b^{i+j+k},g^{-j-k}x+g^{-k}y+z+\ \mysum{I(i+j+k,-j-2k)}xy\ +\ \mysum{I(i+k,-j-k)}xz\ +\ \mysum{I(i+j+k,-k)}yz).
\end{displaymath}
Using Corollary \ref{Cr:iPairFormula} in the third equality, we calculate
\begin{align*} 
	&(b^i,x)(b^j,y)\cdot (b^k,z) = (b^{i+j},g^{-j}x+y+\mysum{I(i+j,-j)}xy)(b^k,z) \\
	&=(b^{i+j+k},g^{-k}(g^{-j}x+y)+g^{-k}\mysum{I(i+j,-j)}xy+z+\mysum{I(i+j+k,-k)}{g^{-j}x+y}{z}\\
	&=(b^{i+j+k},g^{-j-k}x+g^{-k}y+\ \mysum{I(-2k,k)}{g^{-j}x}{y} + z +\ \mysum{I(i+j+k,-j+k)}xy\ +\ \mysum{I(i+k,-j-k)}xz\ +\ \mysum{I(i+j+k,-k)}yz).
\end{align*}
Since 
\begin{displaymath}
	\mysum{I(-2k,k)}{g^{-j}x}{y}\ +\ \mysum{I(i+j+k,-j+k)}xy\ \ =\ \ \mysum{I(-j-2k,-j+k)}xy\  +\  \mysum{I(i+j+k,-j+k)}xy\ \ =\ \ \mysum{I(i+j+k,-j-2k)}xy,
\end{displaymath}
we are through.
\end{proof}

For a loop $Q$, the \emph{commutator subloop} $\com{Q}$ (resp.~\emph{associator subloop} $\asc{Q}$, resp.~\emph{derived subloop} $Q'$) is the smallest normal subloop $N$ such that $Q/N$ is commutative (resp.~associative, resp.~abelian group). Thus $\com{Q}$ (resp.~$\asc{Q}$, resp.~$Q'$) is the normal subloop of $Q$ generated by all commutators (resp.~associators, resp.~commutators and associators).

\begin{proposition}\label{Pr:DerivedSubloop}
Let $Q=\langle b\rangle\ltimes_{(g,\gamma)}X$. Then
\begin{displaymath}
	1\times \langle\img{\gamma}\rangle = \asc{Q}\le\com{Q}=Q'=1\times \langle\img{\gamma}\cup\img{1-g}\rangle.
\end{displaymath}
\end{proposition}
\begin{proof}
It is clear from Lemmas \ref{Lm:Commutator} and \ref{Lm:Associator} that $\com{Q}$, $\asc{Q}$ and $Q'$ are subloops of $1\times X$. Since $X$ is an abelian group, it follows that $\com{Q}$ (resp.~$\asc{Q}$, resp.~$Q'$) is the subloop of $Q$ generated by all commutators (resp.~associators, resp.~commutators and associators), that is, it is not necessary to apply normal closure.

Let $A=\langle \img{1-g}\cup \img{\gamma}\rangle \le X$. We will first show that $\com{Q}=1\times A$. We have $g^{i+1}x-g^ix = g(g^ix)-g^ix\in\img{g-1}\subseteq A$ for all $i$ and therefore $g^jx-x = (g^jx-g^{j-1}x)+(g^{j-1}x-g^{j-2}x)+\cdots+(gx-x)\in A$ for all $j\ge 0$. Since $g^{-j}x-x = -(g^jg^{-j}x-g^{-j}x)$, we have $g^jx-x\in A$ for all $j$. Also, $\mysum{I(i,j)}xy\in A$ as $g$ permutes $\img{\gamma}\subseteq A$. By \eqref{Eq:Commutator},
\begin{displaymath}
	[(b^i,x),(b^j,y)] = \Big(1,-x+g^{-j}x+y-g^{-i}y+\mysum{I(-i,-j)}xy + \mysum{I(-i-j,i+j)}xy\Big).
\end{displaymath}
It follows that $1\times A$ contains all commutators and thus $\com{Q}\le 1\times A$. Conversely, write $\com{Q}=1\times Y$ for some $Y\le X$. We have $[(b^i,x),(b^j,0)] =(1,-x+g^{-j}x)\in \com{Q}$. In particular, $\img{1-g}\subseteq Y$. The commutator formula then implies  $\mysum{I(-i,-j)}xy + \mysum{I(-i-j,i+j)}xy\in Y$. With $j=-i$ we obtain $\mysum{I(-i,i)}xy\in Y$, hence $\mysum{I(-i-j,i+j)}xy\in Y$ and $\mysum{I(-i,-j)}xy\in Y$. Then $\gamma(x,y) = \mysum{I(0,1)}xy\in Y$ and $\img{\gamma}\subseteq Y$. This shows that $1\times A\le\com{Q}$.

The associator formula \eqref{Eq:Associator} implies that $\asc{Q}\le\langle\img{\gamma}\rangle\le\com{Q}$, so $\com{Q}=Q'$. Substituting $i=j=0$, $k=1$ and $z=0$ into \eqref{Eq:Associator} yields $\gamma(x,y) = \mysum{I(0,1)}xy\in\asc{Q}$, so $\langle\img{\gamma}\rangle\le\asc{Q}$.
\end{proof}

\begin{corollary}\label{Cr:NoCML}
Let $Q=\langle b\rangle\ltimes_{(g,\gamma)}X$. If $Q$ is commutative then $Q$ is a group.
\end{corollary}
\begin{proof}
If $Q$ is commutative then $1=\com{Q}$. Since $\asc{Q}\le\com{Q}$ by Proposition \ref{Pr:DerivedSubloop}, $Q$ is a group.
\end{proof}

Recall the parameter $r(g,\gamma)$ of \eqref{Eq:r}, possibly with $r(g,\gamma)=\infty$.

\begin{proposition}\label{Pr:NZ}
Let $Q=\langle b\rangle\ltimes_{(g,\gamma)}X$ and let $r=r(g,\gamma)$.
\begin{align*}
    \nuc{Q} &= \{(b^i,x): x\in\rad{\gamma}\text{ and either }i=0\text{ or }r\text{ divides }i\},\\
    Z(Q) &= \{(b^i,x)\in\nuc{Q}: |g|\text{ divides }i\text{ and }g(x)=x\}.
\end{align*}
\end{proposition}
\begin{proof}
By Lemma \ref{Lm:Associator}, $(b^i,x)\in\nuc{Q}$ if and only $\mysum{I(i+j,i+j+k)}xy + \mysum{I(i+k,i+j+k)}xz + \mysum{I(j+k,i+j+k)}yz = 0$ for all $j,k\in\mathbb Z$ and $y,z\in X$. Suppose that $(b^i,x)\in\nuc{Q}$. Setting $j=0$, $k=1$ and $z=0$ yields $0=\mysum{I(i,i+1)}xy = g^{-i}\gamma(x,y)$, so $x\in\rad{\gamma}$. Then $\mysum{I(j+k,i+j+k)}yz=0$ implies $\mysum{I(0,i)}yz=0$. By Lemma \ref{Lm:r2}, either $i=0$ or $r$ divides $i$. Conversely, if $i=0$ or $r$ divides $i$ then $\mysum{I(j+k,i+j+k)}yz=0$ for all $j,k,y,z$ by Lemma \ref{Lm:r2}. If also $x\in\rad{\gamma}$, we deduce that $(b^i,x)\in\nuc{Q}$.

Suppose that $(b^i,x)\in\nuc{Q}$. By Lemma \ref{Lm:Commutator}, $(b^i,x)\in Z(Q)$ if and only if $-x+g^{-j}x+y-g^{-i}y=0$ for all $j\in\mathbb Z$ and $y\in X$. Suppose that $(b^i,x)\in Z(Q)$. Setting $j=-1$ and $y=0$ yields $x=g(x)$. Then $y=g^{-i}y$ for all $y$ implies that $|g|$ divides $i$. Conversely, if $|g|$ divides $i$ and $x=g(x)$ then $-x+g^{-j}x+y-g^{-i}y=0$ and $(b^i,x)\in Z(Q)$.
\end{proof}

\section{A note on abelian congruences in Moufang loops}\label{Sc:Abelian}

There are two theories of solvability for loops: classical solvability based on the standard definition of solvability from group theory, and congruence solvability based on a general solvability theory from congruence modular varieties. See \cite{FM} for solvability in congruence modular varieties, \cite{StaVojComm} for an introduction to congruence solvability in loops, and \cite{DrVo} for a detailed discussion about solvability in loops.

The two concepts of solvability coincide in groups but congruence solvability is strictly stronger than classical solvability in loops. It is an open question whether the two concepts of solvability coincide in Moufang loops. We showed in \cite{DrVo} that Moufang loops of odd order are congruence solvable, strengthening Glauberman's Odd Order Theorem for Moufang loops \cite{GlaubermanII}.

The principal difficulty is that, unlike in groups, an abelian normal subgroup of a Moufang loop $Q$ need not induce an abelian congruence of $Q$ (see below). For instance, in \cite{DrVo} we constructed a class of nilpotent Moufang loops $Q$ possessing an abelian normal subgroup that does not induce an abelian congruence of $Q$.

A normal subloop $X$ of a loop $Q$ induces the congruence $\rho_X = \{(a,b)\in Q:aX=bX\}$ of $Q$. By results of \cite{StaVojComm} and \cite{Barnes}, the commutator $[\rho,\sigma]_Q$ of congruences $\rho$, $\sigma$ of $Q$ is the congruence of $Q$ generated by
\begin{equation}\label{Eq:CongruenceCommutator}
    \{(T_{b_1}(a),T_{c_1}(a)),\,(L_{b_1,b_2}(a),L_{c_1,c_2}(a)),\,(R_{b_1,b_2}(a),R_{c_1,c_2}(a)): 1\rho a,\,b_1\sigma c_1,\,b_2\sigma c_2\}.
\end{equation}
A congruence $\rho$ of $Q$ is \emph{abelian} if $[\rho,\rho]_Q = \{(a,a):a\in Q\}$.

The following result shows that the Moufang loops afforded by Theorem \ref{Th:AbelianByCyclic} frequently contain an abelian normal subgroup that does not induce an abelian congruence.

\begin{proposition}
Let $Q=\langle b\rangle\ltimes_{(g,\gamma)}X$. Then the abelian normal subgroup $1\times X$ of $Q$ induces an abelian congruence of $Q$ if and only if $Q$ is a group.\end{proposition}
\begin{proof}
If $Q$ is a group then every abelian normal subgroup of $Q$ induces an abelian congruence of $Q$. For the converse, let $\rho$ be the congruence induced by $1\times X$, so that the equivalence classes of $\rho$ are precisely the cosets of $1\times X$ in $Q$. Recall that $T_{a^{-1}}=T_a^{-1}$ in Moufang loops. To show that $\rho$ is not abelian, it therefore suffices to find $x,y,z\in X$ such that $T^{-1}_{(b,y)}(1,x)\ne T^{-1}_{(b,z)}(1,x)$, cf.~\eqref{Eq:CongruenceCommutator}. Now, $T^{-1}_u(v) = u^{-1}vu = v(v^{-1}u^{-1}vu) = v[v,u]$. By Lemma \ref{Lm:Commutator}, we have
\begin{align*}
    T^{-1}_{(b,y)}(1,x) &= (1,x)[(1,x),(b,y)] = (1,x)(1,-x+g^{-1}(x)+y-y+\mysum{I(0,-1)}xy+\mysum{I(-1,1)}xy)\\
    &= (1,x)(1,-x+g^{-1}(x)+\mysum{I(0,1)}xy) = (1,x)(1,-x+g^{-1}(x)+\gamma(x,y))\\
    &= (1,g^{-1}(x)+\gamma(x,y)).
\end{align*}
Suppose that $T^{-1}_{(b,y)}(1,x)=T^{-1}_{(b,z)}(1,x)$ for all $x,y,z\in X$. Then $\gamma(x,y)=\gamma(z,y)$ for all $x,y,z\in X$, hence $y\in\rad{\gamma}$ for every $y\in Q$, and $\img{\gamma}=1$. By Proposition \ref{Pr:DerivedSubloop}, $Q$ is a group.
\end{proof}

\section{Pseudoautomorphisms induced by pseudoautomorphisms}\label{Sc:Pseudo}

We now start working toward a partial converse of Theorem \ref{Th:AbelianByCyclic}. While this section and the next are written more generally than necessary, their main goal is to show that every conjugation in a Moufang loop restricted to an abelian normal subgroup yields a Moufang permutation, cf.~Proposition \ref{Pr:Beta} and Definition \ref{Df:MoufangPerm}.

Throughout this section let $Q$ be a Moufang loop and $X$ a normal subloop of $Q$. For a permutation $f$ on $Q$ and $x\in Q$, define a permutation $\mu_{f,x}$ on $Q$ by
\begin{equation}\label{Eq:Mu}
	\mu_{f,x} = R_x^{-1}f^{-1}R_{f(x)}f.
\end{equation}
Note that if $f(1)=1$ then $\mu_{f,x}(1) = R_x^{-1}f^{-1}f(x) = 1$.

Suppose that $f$ is a pseudoautomorphism of $Q$, cf.~\eqref{Eq:DfPseudo}. By Proposition \ref{Pr:InducedPseudo}, $\mu_{f,x}$ is then also a pseudoautomorphism of $Q$ and we will call $\mu_{f,x}$ the pseudoautomorphism \emph{induced} by $f$ and $x$. In the special case when $f$ is an inner mapping of $Q$, Proposition \ref{Pr:InducedFromInn} shows that the induced pseudoautomorphism $\mu_{f,x}$ acts on $X$ and it gives a necessary condition for the companion of $\mu_{f,x}$ to belong to $X$. In particular, if $x\in X$ and $f\in\inn{Q}$ then the restriction of $\mu_{f,x}$ to $X$ is a pseudoautomorphism of $X$, cf.~Corollary \ref{Cr:I1}. Finally, if $f=T_a$ for some $a\in Q$ then $R_x\mu_{f,x}R_x^{-1}$ can be expressed in terms of $f$ and left translations, cf.~Proposition \ref{Pr:InducedFromT}.

Note that the concept $\mu_{f,x}$ is trivial for pseudoautomorphisms of groups. Every pseudoautomorphism $f$ of a group $Q$ is an automorphism and therefore $\mu_{f,x} =R_x^{-1}(f^{-1}R_{f(x)}f) = R_x^{-1}R_{f^{-1}(f(x))} = R_x^{-1}R_x=\id{Q}$.

\begin{proposition}\label{Pr:InducedPseudo}
Let $Q$ be a Moufang loop, $(c,f)\in\lps{Q}$ and $x\in Q$. Then
\begin{displaymath}
	(x^{-1} f^{-1}(c^{-1}f(x)c),\,\mu_{f,x})\in\lps{Q}.
\end{displaymath}
\end{proposition}
\begin{proof}
Let $\theta_x = (L_x,R_x,L_xR_x)$ and $\vhi = (L_cf,f,L_cf)$. Note that $\theta_x\in\atp{Q}$ since $Q$ is Moufang, and $\vhi\in\atp{Q}$ since $(c,f)\in\lps{Q}$. Hence $\psi = \theta_x^{-1}\vhi^{-1}\theta_{f(x)}\vhi\in \atp{Q}$. The middle coordinate of $\psi$ is equal to $\mu_{f,x} = R_x^{-1}f^{-1}R_{f(x)}f$. Since $f(1)=1$ holds for any pseudoautomorphism $f$ of a Moufang loop, we have $\mu_{f,x}(1)=1$. The first coordinate of $\psi$ maps $1$ to $L_x^{-1}f^{-1}L_c^{-1}(f(x)c) = x^{-1} f^{-1}(c^{-1}f(x)c)$ and we are done by Lemma \ref{Lm:AtpToPseudo}.
\end{proof}

\begin{proposition}\label{Pr:InducedFromInn}
Let $Q$ be a Moufang loop, $X\unlhd Q$ and $x\in Q$. Let $f\in\inn{Q}$ and let $c$ be a companion of $f$ when $f$ is viewed as a pseudoautomorphism. Then $\mu_{f,x}(X)=X$. Moreover, the companion $x^{-1} f^{-1}(c^{-1}f(x)c)$ of $\mu_{f,x}$ belongs to $X$ if and only if $[f(x),c]$ belongs to $X$.
\end{proposition}
\begin{proof}
By Proposition \ref{Pr:InducedPseudo}, $(x^{-1} f^{-1}(c^{-1}f(x)c),\,\mu_{f,x})\in\lps{Q}$. We have $f(X)=X$ since $X\unlhd Q$ and $f\in\inn{Q}$. Hence $\mu_{f,x} = R_x^{-1}f^{-1}R_{f(x)}f$ permutes $X$ iff $Xf(x) = f(Xx)$ iff $c\cdot Xf(x) = cf(Xx)$ iff $cX\cdot f(x) = cf(Xx)$, taking advantage of the normality of $X$ in the last step. The condition $cX\cdot f(x) = cf(Xx)$ holds because $f(X)=X$ and $(c,f)\in\lps{Q}$, that is, $cf(y)\cdot f(x) = cf(yx)$ for every $y\in X$.

The companion $x^{-1} f^{-1}(c^{-1}f(x)c)$ belongs to $X$ iff $c^{-1}f(x)c \in f(xX)$ iff there is $y\in X$ such that $f(x)c = cf(xy) = cf(x) \cdot f(y)$, which is equivalent to $(cf(x))^{-1}f(x)c=[f(x),c]\in X$.
\end{proof}

The condition $[f(x),c]\in X$ of Proposition \ref{Pr:InducedFromInn} certainly holds if $x\in X$ because $[f(x),c] = f(x)^{-1}c^{-1}f(x)c$ and  $f(x) \in X$, $c^{-1}f(x)c\in X$. We therefore have:

\begin{corollary}\label{Cr:I1}
Let $Q$ be a Moufang loop, $X\unlhd Q$ and $x\in X$. Let $f\in\inn{Q}$ and let $c$ be a companion of $f$. Let $\rho$ be the restriction of $\mu_{f,x}$ to $X$. Then $(x^{-1} f^{-1}(c^{-1}f(x)c),\,\rho)\in\lps{X}$.
\end{corollary}

\begin{corollary}\label{Cr:I2}
Let $Q$ be a Moufang loop, $X$ a normal subgroup of $Q$, $x\in X$ and $f\in\inn{Q}$. Then $\mu_{f,x}$ restricts to an automorphism of $X$.
\end{corollary}

\begin{remark}
Corollary \ref{Cr:I2} gives a necessary condition for a semiautomorphism $g$ of a group $X$ to be induced from some inner mapping of some Moufang loop $Q$ that contains $X$ as a normal subgroup. Namely, if there exists $x\in X$ such that $R_x^{-1}g^{-1}R_{g(x)}g$ is not an automorphism of $X$, then $g$ is not so induced.
\end{remark}

Note that $R_x\mu_{f,x}R_x^{-1} = f^{-1} R_{f(x)} fR_x^{-1}$. The right hand side is discussed in the next lemma.

\begin{lemma}\label{Lm:Rewrite}
Let $Q$ be a Moufang loop and suppose that $(c,f)\in\lps{Q}$ satisfies $f(c^{-1}) = c^{-1}$. Then
\begin{equation}\label{Eq:Rewrite1}
	L^{-1}_{f^3(x)} f L_{f^2(x)}f^{-1} = f^{-1} R_{f(x)} fR_x^{-1}
\end{equation}
holds for all $x\in Q$ if and only if
\begin{equation}\label{Eq:Rewrite2}
	 (cx)\m (cxc\m \cdot y) = (f^{-3}(x)c)\m \cdot f^{-3}(x)y
\end{equation}
holds for all $x,y \in Q$.
\end{lemma}
\begin{proof}
Note that \eqref{Eq:Rewrite1} holds for all $x\in Q$ if and only if
\begin{equation}\label{Eq:RewriteAux}
	L\m_x f L_{f\m(x)}f\m = f\m R_{f^{-2}(x)}f R_{f^{-3}(x)}^{-1}
\end{equation}
holds for all $x\in Q$. Applying the left hand side of \eqref{Eq:RewriteAux} to $y\in Q$ yields
\begin{displaymath}
	x\m \cdot f(f\m(x)f\m(y)) \stackrel{\ref{Lm:MoufPseudo}}{=}
	x\m (xc\m\cdot cy) \stackrel{\eqref{Eq:M5}}{=}
	x\m (c\m (cx\cdot y)).
\end{displaymath}
Since $f(c^{-1})=c^{-1}$, \eqref{Eq:LPs} yields $(c,f)^{-1} = (f^{-1}(c^{-1}),f^{-1}) = (c^{-1},f^{-1})\in\lps{Q}$. Let $z = f^{-3}(x)$. Applying the right hand side of \eqref{Eq:RewriteAux} to $y$ then yields
\begin{displaymath}
	f\m(f(yz\m)\cdot f(z)) \stackrel{\ref{Lm:MoufPseudo}}{=}
	(yz\m\cdot c)(c\m z) \stackrel{\eqref{Eq:M5}}{=}
	c((c\m \cdot yz\m)z),
\end{displaymath}
where we have applied Lemma \ref{Lm:MoufPseudo} to $(c^{-1},f^{-1})\in\lps{Q}$. Upon multiplying both sides of \eqref{Eq:RewriteAux} by $c^{-1}$ on the left, we therefore see that \eqref{Eq:RewriteAux} holds if and only if the product
\begin{displaymath}
	c\m(x\m(c\m(cx\cdot y)) \stackrel{\eqref{Eq:M4}}{=}
	c\m(cx)\m \cdot (cx\cdot y) \stackrel{\eqref{Eq:M5}}{=}
	(cx)\m (cxc\m \cdot y)
\end{displaymath}
is equal to
\begin{displaymath}
	(c\m \cdot yz\m)z \stackrel{\eqref{Eq:M6}}{=}
	 c\m z\m\cdot zy = (zc)\m zy,
\end{displaymath}
and we are through.
\end{proof}

\begin{proposition}\label{Pr:InducedFromT}
Let $Q$ be a Moufang loop and $a\in Q$. If $f = T_a$ then \eqref{Eq:Rewrite1} holds for every $x\in Q$.
\end{proposition}
\begin{proof}
The companion of $T_a$ is equal to $c=a^{-3}$. Certainly, $T_a(c^{-1}) = c^{-1}$ by power associativity. Let $z = f^{-3}(x) = T_a^{-3}(x) = a^{-3}xa^3 =cxc\m$. By Lemma \ref{Lm:Rewrite}, it suffices to verify \eqref{Eq:Rewrite2}. Now, $(cx)\m(cxc\m\cdot y) = x\m c\m\cdot zy = (zc)^{-1}\cdot zy$.
\end{proof}

\section{Inducing from a semiautomorphism of an abelian group}\label{Sc:Semi}

We now study the situation when $f$ is a semiautomorphism of an abelian group $(X,+)$ and $\mu_{f,x}$ of \eqref{Eq:Mu} is an automorphism of $(X,+)$.

\begin{proposition}\label{Pr:Semi}
Let $f$ be a semiautomorphism of an abelian group $(X,+)$ such that for every $x\in X$ the permutation
\begin{displaymath}
    \mu_x = \mu_{f,x} = R_x\m f\m  R_{f(x)} f
\end{displaymath}
is an automorphism of $(X,+)$. Let
\begin{displaymath}
    x\oplus y = f\m (f(x)+f(y))
\end{displaymath}
for all $x,y \in X$. Then:
\begin{enumerate}
\item[(i)] $f$ is an isomorphism from $(X,\oplus)$ onto $(X,+)$, and $(X,\oplus)$ is an abelian group with identity element $0$ and inverses $\ominus x = -x$,
\item[(ii)] $x\oplus y = \mu_y(x) + y = x + \mu_x(y)$ for all $x,y\in X$,
\item[(iii)] $\mu_x(x)=x$, $\mu_x R_x = R_x\mu_x$ and $\mu_x = f\m R_{f(x)} f R_x\m$ for all $x\in X$,
\item[(iv)] $x+y = x\oplus\mu_x(y)$ for all $x,y\in X$,
\item[(v)] $\mu_x=\mu_x\m = \mu_{-x}$ and $\mu_x\mu_y = \mu_{x\oplus y}$ for all $x,y \in X$,
\item[(vi)] the mapping $x\mapsto \mu_x$ is an action of $(X,+)$ on $X$ if and only if $\mu_x\in \aut{X,\oplus}$ for every $x\in X$.
\end{enumerate}
\end{proposition}
\begin{proof}
(i) The definition of $\oplus$ says that $f$ is an isomorphism from the groupoid $(X,\oplus)$ onto $(X,+)$. In particular, $(X,\oplus)$ must be an abelian group. Since $f$ is a semiautomorphism of $(X,+)$, we have $f(kx) = kf(x)$ for every integer $k$. Hence $f(0)=0$ is the identity element of $(X,\oplus)$, and $x\oplus(-x) = f^{-1}(f(x)+f(-x)) = f^{-1}(f(x)-f(x))= f^{-1}(0)=0$ shows that $-x$ is the inverse of $x$ in $(X,\oplus)$.

(ii) By definition, $x\oplus y = f^{-1}(f(y)+f(x)) = f^{-1}R_{f(y)}f(x) = R_yR_y^{-1}f^{-1}R_{f(y)}f(x) = R_y(\mu_y(x)) = \mu_y(x)+y$ and therefore also $x\oplus y = y\oplus x = \mu_x(y)+x$ by commutativity of $\oplus$.

(iii) By (ii), $\mu_x(x)+x = x\oplus x = f^{-1}(2f(x)) = f^{-1}(f(2x)) = 2x$, so $\mu_x(x)=x$. Then $\mu_x(y)+x = \mu_x(y)+\mu_x(x) = \mu_x(y+x)$, which says $R_x\mu_x = \mu_xR_x$, and this in turn implies $\mu_x = R_x\mu_xR_x^{-1} = f\m R_{f(x)} f R_x\m$.

(iv) We have $f(x\oplus\mu_x(y)) = f(x)+f(\mu_x(y))$ since $f$ is an isomorphism $(X\oplus)\to(X,+)$. By (iv), $f(\mu_x(y)) = ff\m R_{f(x)} f R_x\m(y) = f(y-x)+f(x)$. Combining, $f(x\oplus\mu_x(y)) = f(x)+f(y-x)+f(x) = f(x+(y-x)+x) = f(x+y)$ since $f$ is a semiautomorphism of $(X,+)$. Hence $x+y = x\oplus\mu_x(y)$.

(v) By (iv) and (ii), $x\oplus\mu_x(y) = x+y = x + \mu_x(\mu_x^{-1}(y)) = x\oplus \mu_x^{-1}(y)$ and so $\mu_x = \mu_x^{-1}$. We then calculate $x+\mu_x(y)+\mu_{x\oplus y}(z) = (x\oplus y)+\mu_{x\oplus y}(z) = (x\oplus y)\oplus z = x\oplus(y\oplus z) = x+\mu_x(y\oplus z) = x+\mu_x(y+\mu_y(z)) = x+\mu_x(y)+\mu_x(\mu_y(z))$, where we used $\mu_x\in\aut{X,+}$ in the last step. We deduce $\mu_x\mu_y = \mu_{x\oplus y}$. Using this and (i), we have $\mu_x\mu_{-x} = \mu_{x\oplus(-x)} = \mu_0 = \id{X}$.

(vi) Recall $x+y = x\oplus\mu_x(y)$ from (iv). Then $x\oplus\mu_x(y\oplus\mu_y(z)) = x\oplus\mu_x(y+z) = x+(y+z) = (x+y)+z = (x+y)\oplus\mu_{x+y}(z) = x\oplus\mu_x(y)\oplus\mu_{x+y}(z)$ and therefore $\mu_x(y\oplus\mu_y(z)) = \mu_x(y)\oplus\mu_{x+y}(z)$ always holds. Note that $\mu_x\in\aut{X,\oplus}$ iff $\mu_x(y\oplus\mu_y(z)) = \mu_x(y)\oplus\mu_x\mu_y(z)$. Suppose that $\mu_x\in\aut{X,\oplus}$, so $\mu_x(y)\oplus\mu_x\mu_y(z) = \mu_x(y\oplus\mu_y(z)) = \mu_x(y)\oplus\mu_{x+y}(z)$, which implies $\mu_x\mu_y = \mu_{x+y}$. Conversely, if $\mu_x\mu_y=\mu_{x+y}$ then
$\mu_x(y)\oplus\mu_x\mu_y(z)= \mu_x(y)\oplus\mu_{x+y}(z) = \mu_x(y\oplus\mu_y(z))$ and $\mu_x\in\aut{X,\oplus}$ follows.
\end{proof}

\begin{corollary}\label{Cr:Kernel}
Under the assumptions of Proposition $\ref{Pr:Semi}$, if $\mu:x\mapsto\mu_x$ is an action of $(X,+)$ on $X$ then $\mu_y(x) - x \in \ker{\mu}$ for all $x,y\in X$.
\end{corollary}
\begin{proof}
If $\mu$ is an action of $(X,+)$ on $X$ then $\mu_{x+y} = \mu_x\mu_y = \mu_{x\oplus y}$ by Proposition \ref{Pr:Semi}. Therefore $\id{X} = \mu_{x\oplus y}\mu\m_{x+y} = \mu_{(x\oplus y)-(x+y)}$. By Proposition \ref{Pr:Semi}(ii), $(x\oplus y)-(x+y) = \mu_y(x)+y-x-y= \mu_y(x)-x$ then lies in the kernel of $\mu$.
\end{proof}

The next statement uses both $L_x$ and $R_x$ to emphasize the connection to Lemma \ref{Lm:Rewrite}. Of course, $L_x = R_x$ for each $x\in X$ since $(X,+)$ is assumed to be commutative.

\begin{lemma}\label{Lm:Auto}
Under the assumptions of Proposition $\ref{Pr:Semi}$, if $L_{f^3(x)}\m f L_{f^2(x)} f\m = f\m R_{f(x)}fR_x\m$ for every $x\in X$ then $f\m \mu_x f = \mu_{f^2(x)}$ and $\mu_x\in
\aut{X,\oplus}$ for every $x\in X$.
\end{lemma}
\begin{proof}
By Proposition \ref{Pr:Semi}(iii), we have $L_{f^3(x)}\m f L_{f^2(x)} f\m = f\m R_{f(x)}fR_x\m = \mu_x$. By \eqref{Eq:SemiEquiv}, $L_{f^3(x)}\m f L_{f^2(x)} f\m = R_{f^3(x)}f R_{f^2(x)}\m f\m$. The right hand side is equal to $f\mu_{f^2(x)}f\m$ by Proposition \ref{Pr:Semi}(iii). Altogether, $\mu_x = f\mu_{f^2(x)} f\m$ for every $x\in X$. Since $f$ is an isomorphism $(X,\oplus)\to(X,+)$ and $\mu_x\in \aut{X,+}$, we have $\mu_{f^2(x)}= f\m \mu_x f\in \aut{X,\oplus}$.
\end{proof}

\begin{proposition}\label{Pr:Beta}
Let $(X,+)$ be a normal abelian subgroup of a Moufang loop $(Q,\cdot)$. Let $a\in Q$ and let $f$ be the restriction of $T_a$ to $X$. Then for every $x\in X$ the permutation $\mu_{f,x} = \mu_x$ defined by \eqref{Eq:Mu} is an automorphism of $(X,+)$ and $\mu:x\mapsto\mu_x$ is an action of $(X,+)$ on $X$. Furthermore, $f$ is a Moufang permutation on $(X,+)$, cf.~Definition $\ref{Df:MoufangPerm}$.
\end{proposition}
\begin{proof}
The mappings $\mu_x$ are automorphisms of $(X,+)$ by Corollary \ref{Cr:I2}. All assumptions of Proposition \ref{Pr:Semi} are therefore satisfied. By Proposition \ref{Pr:InducedFromT}, all assumptions of Lemma \ref{Lm:Auto} are also satisfied. Hence $\mu$ is an action of $(X,+)$ on $X$ by Lemma \ref{Lm:Auto} and Proposition \ref{Pr:Semi}(vi).

Let $\beta$ be defined as in \eqref{Eq:BetaDef}, that is, $\beta(x,y) = f^{-1}(f(x)+f(y))-x-y$. Obviously, $\beta$ is symmetric. Let $x\oplus y = f^{-1}(f(x)+f(y))$ be as in Proposition \ref{Pr:Semi} and recall that $x\oplus y = \mu_x(y)+x$ so that $\beta(x,y) = x\oplus y - x - y = \mu_x(y) - y$. Since $\mu_x\in\aut{X,+}$, $\beta$ is biadditive. Finally, $\mu_x(x)=x$ of Proposition \ref{Pr:Semi}(iii) implies that $\beta$ is alternating.

We have $\beta(\beta(x,y),z) = \mu_{\beta(x,y)}(z)-z = \mu_{\mu_x(y)-y}(z)-z$. By Corollary \ref{Cr:Kernel}, $\mu_x(y)-y$ is in the kernel of $\mu$, so $\beta(\beta(x,y),z) = 0$, which is \eqref{Eq:BetaRad}.

Rewriting $\beta(\beta(x,y),z)=0$ from the definition yields $f(\beta(x,y)+z) = f(\beta(x,y))+f(z)$. By Lemma \ref{Lm:Auto}, $\mu_x f = f\mu_{f^2(x)}$ for all $x\in X$. Therefore $\beta(x,f(y)) + f(y) = \mu_xf(y) = f\mu_{f^2(x)}(y) = f(\beta(f^2(x),y) + y) = f(\beta(f^2(x),y)) + f(y)$. Substituting $f(x)$ for $x$ then yields $\beta(f(x),f(y)) = f(\beta(f^3(x),y))$, which is \eqref{Eq:BetaProp}. We have proved that $f$ is a Moufang permutation on $(X,+)$.
\end{proof}

\section{Moufang permutations and their properties}\label{Sc:MPP}

In this section we first investigate properties of Moufang permutations on an abelian group $(X,+)$. We then prove that powers of Moufang permutations are Moufang permutations and we calculate the associated biadditive mappings, cf.~Proposition \ref{Pr:fiBeta}. Several statements and proofs in this section resemble those of Section \ref{Sc:CPP} but since neither section is a special case of the other, we present the arguments in full.

It is easy to see from Definition \ref{Df:MoufangPerm} that the identity mapping on $X$ is a Moufang permutation. Also note that if $(f,\beta)$ is a Moufang pair on $(X,+)$ then $\beta=0$ if and only if $f$ is an automorphism of $(X,+)$.

\begin{lemma}\label{Lm:RadImg}
Let $(f,\beta)$ be a Moufang pair on $(X,+)$ and let $i$ be an integer. Then:
\begin{enumerate}
\item[(i)] $f(0)=0$, $f(2x)=2f(x)$ and $f(-x)=-f(x)$ for all $x\in X$,
\item[(ii)] $\img{\beta}\subseteq\rad{\beta}$, $2X\subseteq\rad{\beta}$ and $2\img{\beta}=0$,
\item[(iii)] $f^i$ permutes both $\rad{\beta}$ and $\img{\beta}$,
\item[(iv)] $f^i(x+y)=f^i(x)+f^i(y)$ whenever $\{x,y\}\cap\rad{\beta}\ne\emptyset$,
\item[(v)] $f^i$ restricts to an automorphism of $\rad{\beta}$.
\end{enumerate}
\end{lemma}
\begin{proof}
(i) Recall the basic properties of mappings $X\times X\to X$ gathered in Subsection \ref{Ss:Mappings}. Since $\beta$ is alternating, we have $0=\beta(x,x) = f^{-1}(f(x)+f(x))-x-x = f^{-1}(2f(x))-2x$ and so $2f(x)=f(2x)$. In particular, $2f(0)=f(2\cdot 0)=f(0)$ and $f(0)=0$ follows. Biadditivity then implies $f^{-1}(0)=0=-\beta(x,x) = \beta(x,-x) = f^{-1}(f(x)+f(-x))$, so $f(x)+f(-x)=0$.

(ii) The condition $\img{\beta}\subseteq\rad{\beta}$ is a restatement of \eqref{Eq:BetaRad}. As pointed out in Subsection \ref{Ss:Mappings}, $\beta(2x,y)=2\beta(x,y)=0$.

(iii) The following conditions are equivalent for $y\in X$: $y\in\rad{\beta}$, $\beta(x,y)=0$ for all $x\in X$, $\beta(f^3x,y)=0$ for all $x\in X$, $f^{-1}\beta(fx,fy)=0$ for all $x\in X$ (by \eqref{Eq:BetaProp}), $\beta(fx,fy)=0$ for all $x\in X$ (since $f(0)=0$ by (i)), $\beta(x,fy)=0$ for all $x\in X$, $f(y)\in\rad{\beta}$. Similarly, $z\in\img{\beta}$ iff $z=\beta(f^3x,y)$ for some $x,y\in X$ iff $z=f^{-1}\beta(fx,fy)$ for some $x,y\in X$ iff $f(z)\in\img{\beta}$.

(iv) The defining condition \eqref{Eq:BetaDef} can be rewritten as $f(x+y+\beta(x,y))=f(x)+f(y)$. Suppose without loss of generality that $x\in\rad{\beta}$. Then $\beta(x,y)=0$ and $f(x+y)=f(x)+f(y)$ follows. By (iii), $f^i(x)\in\rad{\beta}$ for any integer $i$. By induction, we then have $f^{i+1}(x+y) = f^i(f(x+y)) = f^i(f(x)+f(y)) = f^i(f(x))+f^i(f(y)) = f^{i+1}(x)+f^{i+1}(y)$ for any $i>0$. Continuing with $i>0$, we have $f^i(f^{-i}(x)+f^{-i}(y)) = f^i(f^{-i}(x))+f^i(f^{-i}(y)) = x+y$ and hence $f^{-i}(x+y)=f^{-i}(x)+f^{-i}(y)$. Part (v) follows from (iii) and (iv).
\end{proof}

\begin{lemma}\label{Lm:Prop}
Let $f$ be a Moufang pair on $(X,+)$. Then
\begin{align}
	\beta(f^ix,f^iy) &= f^i\beta(f^{3i}x,y)\label{Eq:Prop1},\\
	\beta(f^{3i}x,f^{3j}y) &= f^{-3(i+j)}\beta(x,y)\label{Eq:Prop2}
\end{align}
for all integers $i$, $j$.
\end{lemma}
\begin{proof}
Let us first prove \eqref{Eq:Prop1} for all nonnegative integers $i$ by induction on $i$. There is nothing to show for $i=0$. For $i=1$, \eqref{Eq:Prop1} becomes the defining condition \eqref{Eq:BetaProp}. Suppose that \eqref{Eq:Prop1} holds for some $i\ge 0$. Then
\begin{align*}
	\beta(f^{i+1}x,f^{i+1}y) &= \beta(f^ifx,f^ify) = f^i\beta(f^{3i}fx,fy)\\
	&= f^i\beta(ff^{3i}x,fy) \stackrel{\eqref{Eq:BetaProp}}{=} f^if\beta(f^3f^{3i}x,y) = f^{i+1}\beta(f^{3(i+1)}x,y).
\end{align*}
Still with a nonnegative integer $i$, we then also have
\begin{displaymath}
	f^i\beta(f^{-i}x,f^{-i}y) = f^i\beta(f^{3i}f^{-4i}x,f^{-i}y) \stackrel{\eqref{Eq:Prop1}}{=} \beta(f^if^{-4i}x,y) = \beta(f^{-3i}x,y),
\end{displaymath}
which implies $\beta(f^{-i}x,f^{-i}y) = f^{-i}\beta(f^{-3i}x,y)$, finishing the proof of \eqref{Eq:Prop1}.

By symmetry of $\beta$, we then have for any integer $i$
\begin{displaymath}
	f^i\beta(f^{3i}x,y) \stackrel{\eqref{Eq:Prop1}}{=} \beta(f^ix,f^iy) = \beta(f^iy,f^ix) \stackrel{\eqref{Eq:Prop1}}{=} f^i\beta(f^{3i}y,x) = f^i\beta(x,f^{3i}y).
\end{displaymath}
Applying $f^{-i}$ to both sides yields
\begin{equation}\label{Eq:PropAux}
	\beta(f^{3i}x,y) = \beta(x,f^{3i}y).
\end{equation}
For \eqref{Eq:Prop2}, first note that
\begin{displaymath}
	f^{-3i}\beta(x,y) = f^{-3i}\beta(f^{3(-3i)}f^{9i}x,y) \stackrel{\eqref{Eq:Prop1}}{=} \beta(f^{-3i}f^{9i}x,f^{-3i}y)
	= \beta(f^{6i}x,f^{-3i}y) \stackrel{\eqref{Eq:PropAux}}{=} \beta(f^{3i}x,y),
\end{displaymath}
which is \eqref{Eq:Prop2} for the special case of $j=0$. For any $j$, we then have
\begin{displaymath}
	\beta(f^{3i}x,f^{3j}y) \stackrel{\eqref{Eq:PropAux}}{=}  \beta(f^{3(i+j)}x,y) = f^{-3(i+j)}\beta(x,y),
\end{displaymath}
finishing the proof.
\end{proof}

We proceed to show that if $f$ is a Moufang permutation on $(X,+)$ then $f^i$ is also a Moufang permutation on $(X,+)$.

\begin{lemma}\label{Lm:fiBeta}
Let $(f,\beta)$ be a Moufang pair on $(X,+)$. Then for every integer $i$, the mapping $f^i\beta:X\times X\to X$ is symmetric, alternating and biadditive, and it satisfies $\img{f^i\beta} = \img{\beta}\subseteq\rad{\beta} = \rad{f^i\beta}$.
\end{lemma}
\begin{proof}
By Lemma \ref{Lm:RadImg}, $f^i$ is an automorphism of $\rad{\beta}$, $\img{\beta}\subseteq\rad{\beta}$ and $f^i$ permutes $\img{\beta}$. Therefore $\img{f^i\beta}=\img{\beta}$, $f^i\beta$ is symmetric and biadditive (as $\beta$ is symmetric and biadditive), and also alternating (since $\beta$ is alternating and $f(0)=0$). Using $f(0)=0$ again, we see that $f^i\beta(x,y)=0$ if and only if $\beta(x,y)=0$. Hence $\rad{f^i\beta}=\rad{\beta}$.
\end{proof}

\begin{lemma}\label{Lm:Inverse}
If $(f,\beta)$ is a Moufang pair on $(X,+)$ then $(f^{-1},f^3\beta)$ is also a Moufang pair on $(X,+)$.
\end{lemma}
\begin{proof}
Recall that $f(-x)=-f(x)$ and $2\img{\beta}=0$, by Lemma \ref{Lm:RadImg}. By \eqref{Eq:BetaDef} and Lemma \ref{Lm:RadImg}, we have $f(x)+f(y)=f(x+y+\beta(x,y)) = f(x+y)+f\beta(x,y) = f(x+y)-f\beta(x,y)$, so $f\beta(x,y) = f(x+y)-f(x)-f(y)$ for any $x,y\in X$. Using this in the last step of the following calculation, we see that
\begin{displaymath}
    f^3\beta(x,y) \stackrel{\eqref{Eq:Prop2}}{=}
    \beta(f^{-3}x,y) \stackrel{\eqref{Eq:Prop1}}{=}
    f\beta(f\m x,f\m y) = f(f^{-1} x+f^{-1}y)-x-y,
\end{displaymath}
which is the analog of \eqref{Eq:BetaDef} for $(f^{-1},f^3\beta)$. Moreover,
\begin{displaymath}
	f^3\beta(f\m x,f\m y) \stackrel{\eqref{Eq:Prop1}}{=} f^3f^{-1}\beta(f^{-3}x,y) = f^2\beta(f^{-3} x,y) = f^{-1}f^3\beta((f^{-1})^3x,y),
\end{displaymath}
which is the analog of \eqref{Eq:BetaProp}. By Lemma \ref{Lm:fiBeta}, $f^3\beta$ is alternating and biadditive.  Since $f$ permutes $\img{\beta}\subseteq\rad{\beta}$ and $f(0)=0$ by Lemma \ref{Lm:RadImg}, we have $f^3\beta(f^3\beta(x,y),z) = f^3(0)=0$, the analog of \eqref{Eq:BetaRad}.
\end{proof}

Recall the intervals $I(i,j)$ of \eqref{Eq:Interval}.

\begin{proposition}\label{Pr:iBetaDef}
Let $(f,\beta)$ be a Moufang pair on $(X,+)$. Then for every integer $i$ we have
\begin{equation}\label{Eq:iBetaDef}
    f^{-i}(f^i(x)+f^i(y)) = x+y+\lcsum{k\in I(0,i)}f^{-3k}\beta(x,y)
\end{equation}
for every $x,y\in X$.
\end{proposition}
\begin{proof}
Let us first prove \eqref{Eq:iBetaDef} for all $i\ge 0$, a situation in which $I(0,i)=\{0,1,\dots,i-1\}$. The case $i=0$ is clear since $I(0,0)=\emptyset$, and the case $i=1$ is \eqref{Eq:BetaDef}. If \eqref{Eq:iBetaDef} holds for some $i\ge 0$ then
\begin{displaymath}
    f^{-i-1}(f^{i+1}x+f^{i+1}y) = f^{-1}f^{-i}(f^ifx+f^ify) = f^{-1}\Big(fx+fy+\lcsum{k\in I(0,i)}f^{-3k}\beta(fx,fy)\Big).
\end{displaymath}
By Lemma \ref{Lm:RadImg}, $\img{\beta}\subseteq\rad{\beta}$ and $f$ permutes $\rad{\beta}$, which implies that every summand of $\sum f^{-3k}\beta(fx,fy)$ as thus also the entire sum are elements of $\rad{\beta}$. By Lemma \ref{Lm:RadImg}(iv), we can the continue the above calculation as
\begin{displaymath}
    f^{-1}(fx+fy) + f^{-1}\lcsum{k\in I(0,i)}f^{-3k}\beta(fx,fy) = f^{-1}(fx+fy)+\lcsum{k\in I(0,i)}f^{-3k-1}\beta(fx,fy).
\end{displaymath}
Now, $f^{-1}(fx+fy) = x+y+\beta(x,y)$ by \eqref{Eq:BetaDef}. Also,  $\beta(fx,fy) = f\beta(f^3x,y) = f^{-2}\beta(x,y)$ by \eqref{Eq:BetaProp} and \eqref{Eq:Prop2}, so
\begin{displaymath}
	\lcsum{k\in I(0,i)}f^{-3k-1}\beta(fx,fy) = \lcsum{k\in I(0,i)}f^{-3k-1}f^{-2}\beta(x,y) = \lcsum{k\in I(0,i)}f^{-3(k+1)}\beta(x,y) = \sum_{k=1}^i f^{-3k}\beta(x,y).
\end{displaymath}
 Altogether,
\begin{displaymath}
	f^{-i-1}(f^{i+1}x+f^{i+1}y)  = x+y+\beta(x,y) + \sum_{k=1}^{i}f^{-3k}\beta(x,y) = x+y+\sum_{k=0}^{i}f^{-3k}\beta(x,y),
\end{displaymath}
which is \eqref{Eq:iBetaDef} for $i+1$.

Let now $i<0$ so that  $I(0,i) = \{-i,-i+1,\dots,-1\}$. By Lemma \ref{Lm:Inverse}, $(f^{-1},f^3\beta)$ is a Moufang pair. For $i=-1$, \eqref{Eq:iBetaDef} reduces to $f(f^{-1}(x)+f^{-1}(y)) = x+y+f^3\beta(x,y)$, which is the analog of \eqref{Eq:BetaDef} for $(f^{-1},f^3\beta)$. If \eqref{Eq:iBetaDef} holds for some $i<0$ then
\begin{displaymath}
	f^{-i+1}(f^{i-1}x+f^{i-1}y) = ff^{-i}(f^if^{-1}x+f^if^{-1}y) = f\Big(f^{-1}x{+}f^{-1}y{+}\lcsum{k\in I(0,i)}f^{-3k}\beta(f^{-1}x,f^{-1}y)\Big),
\end{displaymath}
which by Lemma \ref{Lm:RadImg} can be further written as
\begin{displaymath}
	f(f^{-1}x+f^{-1}y) + f\lcsum{k\in I(0,i)}f^{-3k}\beta(f^{-1}x,f^{-1}y) = f(f^{-1}x+f^{-1}y)+\lcsum{k\in I(0,i)}f^{-3k+1}\beta(f^{-1}x,f^{-1}y).
\end{displaymath}
Now, $f(f^{-1}x+f^{-1}y) = x+y+f^3\beta(x,y)$ by \eqref{Eq:BetaDef} for the Moufang pair $(f^{-1},f^3\beta)$. Also, $\beta(f^{-1}x,f^{-1}y) = f^{-1}\beta(f^{-3}x,y) = f^2\beta(x,y)$ by \eqref{Eq:Prop1} and \eqref{Eq:Prop2}, so
\begin{displaymath}
	\lcsum{k\in I(0,i)}f^{-3k+1}\beta(f^{-1}x,f^{-1}y) = \lcsum{k\in I(0,i)}f^{-3k+1}f^2\beta(x,y) =  \lcsum{k\in I(0,i)}f^{-3(k-1)}\beta(x,y) = \csum{k=-(i-1)}{-2}f^{-3k}\beta(x,y).
\end{displaymath}
Altogether,
\begin{displaymath}
	f^{-i+1}(f^{i-1}x+f^{i-1}y)   = x+y+f^3\beta(x,y) + \csum{k=-(i-1)}{-2}f^{-3k}\beta(x,y) = x+y+\csum{k=-(i-1)}{-1}f^{-3k}\beta(x,y),
\end{displaymath}
which is \eqref{Eq:iBetaDef} for $i-1$.
\end{proof}

\begin{proposition}\label{Pr:fiBeta}
Let $(f,\beta)$ be a Moufang pair on $(X,+)$ and $i\in\mathbb Z$. Let
\begin{displaymath}
	\beta_i = \lcsum{k\in I(0,i)}{f^{-3k}\beta}.
\end{displaymath}
Then $(f^i,\beta_i)$ is a Moufang pair on $(X,+)$. Furthermore, $\img{\beta_i}\subseteq\img{\beta}\subseteq\rad{\beta}\subseteq\rad{\beta_i}$.
\end{proposition}
\begin{proof}
By Lemma \ref{Lm:fiBeta}, $\beta_i$ is a symmetric alternating biadditive mapping such that $\img{\beta_i}\subseteq \img{\beta}\subseteq \rad{\beta} \subseteq \rad{\beta_i}$. The analog of \eqref{Eq:BetaDef} for $(f^i,\beta_i)$ holds by Proposition \ref{Pr:iBetaDef}. By \eqref{Eq:Prop1} and Lemma \ref{Lm:RadImg}, we have
\begin{displaymath}
	\beta_i(f^ix,f^iy) = \lcsum{k\in I(0,i)}f^{-3k}\beta(f^ix,f^iy) = f^i \lcsum{k\in I(0,i)}f^{-3k}\beta(f^{3i}x,y) = f^i\beta_i(f^{3i}x,y),
\end{displaymath}
which is the analog of \eqref{Eq:BetaProp}. The analog of \eqref{Eq:BetaRad} is satisfied by a similar argument as in the proof of Lemma \ref{Lm:Inverse}.
\end{proof}

\begin{corollary}\label{Cr:iFormula}
Let $(f,\beta)$ be a Moufang pair on $(X,+)$ and let $\beta_i$ be defined as in Proposition $\ref{Pr:fiBeta}$. Then
\begin{equation}\label{Eq:iFormula}
	f^{-i}(x+y)=f^{-i}(x) + f^{-i}(y) + f^{2i}\beta_i(x,y).
\end{equation}
for every $x,y\in X$ and $i\in\mathbb Z$.
\end{corollary}
\begin{proof}
By \eqref{Eq:BetaDef} for the Moufang pair $(f^i,\beta_i)$, we have
\begin{displaymath}
	f^{-i}(x+y)= f^{-i}(f^if^{-i}(x) + f^if^{-i}(y)) =f^{-i}(x) + f^{-i}(y) + \beta_i(f^{-i}x,f^{-i}y).
\end{displaymath}
By \eqref{Eq:BetaProp} and \eqref{Eq:Prop2} for $(f^i,\beta_i)$, we have
\begin{displaymath}
	\beta_i(f^{-i}x,f^{-i}y) = f^{-i}\beta_i(f^{-3i}x,y) = f^{-i}f^{3i}\beta_i(x,y) = f^{2i}\beta(x,y).\qedhere
\end{displaymath}
\end{proof}

Finally, here is the relation between Moufang pairs and construction pairs announced in the introduction.

\begin{proposition}\label{Pr:MPCP}
Let $(f,\beta)$ be a Moufang pair on $(X,+)$. Then $(f^3,\beta)$ is a construction pair on $(X,+)$.
\end{proposition}
\begin{proof}
By definition, $\beta$ is alternating and biadditive. It is clear from \eqref{Eq:BetaDef} that $\beta$ is also symmetric. Let $(g,\gamma)=(f^3,\beta)$. Then \eqref{Eq:BetaRad} becomes \eqref{Eq:PairRad}. By \eqref{Eq:Prop2}, $g^{-1}\gamma(x,y) = f^{-3}\beta(x,y) = \beta(f^3x,y) = \gamma(gx,y)$, which is \eqref{Eq:PairProp}. Finally, by \eqref{Eq:iBetaDef} for $i=3$, we have $g^{-1}(gx+gy) =f^{-3}(f^3x+f^3y) = x+y+\sum_{k=0}^2f^{-3k}\beta(x,y) = x+y + \gamma(x,y) + g^{-1}\gamma(x,y) + g^{-2}\gamma(x,y)$, which is \eqref{Eq:PairDef}.
\end{proof}

\section{Moufang loops constructed from Moufang permutations}\label{Sc:MainResults}

In this section we consider the Moufang loops constructed from Moufang permutation and we prove most of the main results.

\begin{theorem}\label{Th:MoufangPair}
Let $(X,+)$ be an abelian normal subgroup of a Moufang loop $(Q,\cdot)$ and let $a\in Q$. Denote by $f$ the restriction
of $T_a$ to $X$. Then $f$ is a Moufang permutation on $(X,+)$ with associated biadditive mapping $\beta$ defined by \eqref{Eq:BetaDef}. Moreover,
\begin{align}
    a^{3i}x\cdot ya^{3j} &= a^{3(i+j)}\Big(f^{-3j}(x)+f^{-3j}(y) + \lcsum{k\in I(-2j,i)}f^{-3k}\beta(x,y)\Big),\label{Eq:F1}\\
    a^{3i}x\cdot a^{3j}y &= a^{3(i+j)}\Big(f^{-3j}(x) + y + \lcsum{k\in I(i+j,-j)}f^{-3k}\beta(x,y)\Big)\label{Eq:F2}
\end{align}
for all integers $i,j$ and all $x,y\in Q$.
\end{theorem}
\begin{proof}
By Proposition \ref{Pr:Beta}, $f$ is a Moufang permutation on $(X,+)$ with associated biadditive mapping $\beta$. Let $\delta = i-j$ and $b=a^3$. By \eqref{Eq:V2}, we have
\begin{displaymath}
	b^ix\cdot yb^j = b^{i+j}f^{-i-2j}(f^{i-j}x+f^{i-j}y) = b^{i+j}f^{-3j}f^{-\delta}(f^\delta x+f^\delta y).
\end{displaymath}
Since $(f^\delta,\beta_\delta)$ is a Moufang pair by Proposition \ref{Pr:fiBeta}, we obtain
\begin{displaymath}
	b^ix\cdot yb^j = b^{i+j}f^{-3j}f^{-\delta}(f^\delta x+f^\delta y) = b^{i+j}f^{-3j}(x+y+\beta_\delta(x,y)).
\end{displaymath}
By Lemma \ref{Lm:RadImg} and \eqref{Eq:iFormula}, $f^{-3j}(x+y+\beta_\delta(x,y))$ is equal to
\begin{displaymath}
	f^{-3j}(x+y)+f^{-3j}\beta_\delta(x,y) = f^{-3j}x+f^{-3j}y+f^{6j}\beta_{3j}(x,y)+f^{-3j}\beta_\delta(x,y).
\end{displaymath}
By Lemma \ref{Lm:RadImg} and Lemma \ref{Lm:Interval}(ii), we have
\begin{displaymath}
    f^{6j}\beta_{3j} = f^{6j}\lcsum{k\in I(0,3j)}f^{-3k}\beta  = \lcsum{k\in I(0,3j)}f^{-3(k-2j)}\beta  = \lcsum{k\in I(0,3j)-2j}f^{-3k}\beta = \lcsum{k\in I(-2j,j)}f^{-3k}\beta
\end{displaymath}
and
\begin{displaymath}
    f^{-3j}\beta_\delta  =  f^{-3j}\lcsum{k\in I(0,\delta)}f^{-3k}\beta  = \lcsum{k\in I(0,i-j)}f^{-3(k+j)}\beta  = \lcsum{k\in I(0,i-j)+j}f^{-3k}\beta  = \lcsum{k\in I(j,i)}f^{-3k}\beta.
\end{displaymath}
Since $2\img{\beta}=0$ by Lemma \ref{Lm:RadImg}, Lemma \ref{Lm:Interval} implies
\begin{displaymath}
    f^{6j}\beta_{3j} + f^{-3j}\beta_\delta =\lcsum{k\in I(-2j,j)}f^{-3k}\beta + \lcsum{k\in I(j,i)}f^{-3k}\beta\  =\  \lcsum{k\in I(-2j,j)\oplus I(j,i)}f^{-3k}\beta
 	\ =\  \lcsum{k\in I(-2j,i)}f^{-3k}\beta,
\end{displaymath}
finishing the proof of \eqref{Eq:F1}.

For \eqref{Eq:F2}, note that $T_b=f^3$, so $b^ix\cdot b^jy = b^ix\cdot b^jyb^{-j}b^j = b^ix\cdot f^{3j}(y)b^j$. Now \eqref{Eq:F1} yields
\begin{displaymath}
    b^ix\cdot b^jy = b^{i+j}\Big(f^{-3j}(x)+y+\lcsum{k\in I(-2j,i)}f^{-3k}\beta(x,f^{3j}y)\Big).
\end{displaymath}
By \eqref{Eq:Prop2} and Lemma \ref{Lm:Interval}, the sum in the above formula is equal to
\begin{displaymath}
	\sum_{k\in I(-2j,i)}f^{-3(k+j)}\beta(x,y) = \lcsum{k\in I(-2j,i)+j}f^{-3k}\beta(x,y) =
	\lcsum{k\in I(-j,i+j)}f^{-3k}\beta(x,y) =  \lcsum{k\in I(i+j,-j)}f^{-3k}\beta(x,y),
\end{displaymath}
finishing the proof.
\end{proof}

\begin{corollary}\label{Cr:MultMP}
Under the assumptions of Theorem $\ref{Th:MoufangPair}$, if $Q=\langle a^3\rangle X$ then each of the formulas \eqref{Eq:F1}, \eqref{Eq:F2} fully describes the multiplication in $Q$.
\end{corollary}

Recall the loops $C\ltimes_{(g,\gamma)} X$ of Definition \ref{Df:Ext} that are Moufang by Theorem \ref{Th:AbelianByCyclic}. Per our convention, whenever we write $C\ltimes_{(g,\gamma)} X$, we assume that all assumptions of Definition \ref{Df:Ext} are satisfied. Also recall the parameter $r(g,\gamma)$ of \eqref{Eq:r} that appears in Definition \ref{Df:Ext}.

\begin{lemma}\label{Lm:AssumptionsOK}
Let $(X,+)$ be a $3$-divisible abelian group and $C$ a finite cyclic group. Let $f$ be a Moufang permutation on $(X,+)$ with associated biadditive mapping $\beta$ such that $|f^3|$ divides $|C|$. Then $(g,\gamma)=(f^3,\beta)$ is a construction pair that satisfies all assumptions of Definition $\ref{Df:Ext}$.
\end{lemma}
\begin{proof}
By Proposition \ref{Pr:MPCP}, $(g,\gamma)=(f^3,\beta)$ is a construction pair on $(X,+)$. Let $n=|C|$. By our assumption, $|g|$ divides $n$. Let $r=r(g,\gamma)$. Since $X$ is $3$-divisible, it contains no elements of order $3$. Then certainly $\rad{\gamma}$ contains no elements of order $3$. Since $|g|<\infty$, Lemma \ref{Lm:r3} implies that $r$ divides $|g|$. Thus $r$ divides $n$ and all assumptions of Definition \ref{Df:Ext} are satisfied.
\end{proof}

\begin{theorem}\label{Th:HomImg}
Let $Q=CX$ be a $3$-divisible Moufang loop, where $(X,+)$ is a normal abelian subgroup of $Q$ and $C=\langle a^3\rangle$ for some $a\in C$.  Let $f$ be the restriction of $T_a$ to $X$ and let $\beta$ be the associated biadditive mapping. Then:
\begin{enumerate}
\item[(i)] $(g,\gamma)=(f^3,\beta)$ is a construction pair satisfying all assumptions of Definition $\ref{Df:Ext}$ and  the Moufang loop $M = \langle a^3\rangle\ltimes_{(g,\gamma)}X$ is well defined,
\item[(ii)] $\varphi:M\to Q$ defined by $\varphi(a^{3i},x)\mapsto a^{3i}x$ is a surjective homomorphism,
\item[(iii)] the kernel of $\varphi$ intersects both $C\times 0$ and $1\times X$ trivially.
\end{enumerate}
\end{theorem}
\begin{proof}
(i) Since $C$ is $3$-divisible, it is finite, say of order $n$, By Theorem \ref{Th:MoufangPair}, the restriction $f$ of $T_a$ to $X$ is a Moufang permutation on $(X,+)$ and the multiplication in $Q$ is fully determined by \eqref{Eq:F2}. By Proposition \ref{Pr:MPCP}, $(g,\gamma) = (f^3,\beta)$ is a construction pair. On $X$, we have  $f^{3n} = (T_a)^{3n} = (T_{a^{3n}})  = T_1 = \id{X}$, so $|f^3|$ divides $n$. By Lemma \ref{Lm:AssumptionsOK}, all assumptions of Definition \ref{Df:Ext} are satisfied. By Theorem \ref{Th:AbelianByCyclic}, $M=\langle a^3\rangle \ltimes_{(f^3,\beta)} X$ is a well-defined Moufang loop with multiplication \eqref{Eq:MultInner}.

(ii) and (iii): By \eqref{Eq:F2} and \eqref{Eq:MultInner}, we have
\begin{align*}
	&\varphi(a^{3i},x)\varphi(a^{3j},y) = a^{3i}x\cdot a^{3j}y = a^{3(i+j)}(f^{-3j}(x) + y + \lcsum{k\in I(i+j,-j)}f^{-3k}\beta(x,y)) \\
	&= \varphi(a^{3(i+j)},f^{-3j}(x)+y+\lcsum{k\in I(i+j,-j)}f^{-3k}(\beta(x,y))) = \varphi((a^{3i},x)(a^{3j},y)),
\end{align*}
so $\varphi$ is a homomorphism, clearly surjective. The kernel of $\varphi$ consists of all $(a^{3i},x)\in M$ with $0\le i<n$ such that $a^{3i}x=1$. If $x=1$, we get $a^{3i}=1$ and and hence $a^3=1$, $a=1$. If $a^{3i}=1$, we get $x=1$.
\end{proof}

\begin{proposition}\label{Pr:Kernel}
Let $S$ be a nonempty subset of $Q = C\ltimes_{(g,\gamma)}X = \langle b\rangle\ltimes_{(g,\gamma)}X$. Then $S$ is a normal subloop of $Q$ that intersects both $1\times X$ and $C\times 0$ trivially if and only if there are $k\in\mathbb Z$ and $z\in X$ such that $S=\langle (b^k,z)\rangle = \{(b^{ki},iz):i\in\mathbb Z\}$, $|b^k|=|z|$ and $(b^k,z)\in Z(Q)$.
\end{proposition}
\begin{proof}
If $(b^k,z)\in Z(Q)$ then $g(z)=z\in\rad{\gamma}$ by Proposition \ref{Pr:NZ} and thus $(b^k,z)^i = (b^{ki},iz)$ for all $i\in\mathbb Z$. The converse implication is now clear since every central subloop is normal and the condition $|b^k|=|z|$ guarantees that $S$ intersects both $1\times X$ and $C\times 0$ trivially.

Let us prove the direct implication. Suppose that $S\unlhd Q$ intersects both $1\times X$ and $C\times 0$ trivially. We claim that for every $c\in C$ there is at most one $x\in X$ such that $(c,x)\in S$. Suppose that $(c,x),(c,y)\in S$ for some $x,y\in X$. Then $(c,x)(c,y)^{-1}\in S$ and $(c,x)(c,y)^{-1}=(1,u)$ for some $u\in X$. Since $S$ intersects $1\times X$ trivially, we have $u=0$ and $x=y$.

Hence there are $k\in\mathbb Z$ and $z\in X$ such that $S=\langle (b^k,z)\rangle$. We have $(b^k,z)\in S$ and also $T_{(b^i,x)}(b^k,z)\in S$ since $S\unlhd Q$. But $T_{(b^i,x)}(b^k,z) = (b^i,x)(b^k,z)\cdot (b^i,z)^{-1} = (b^k,u)$ for some $u\in X$ and it follows that $T_{(b^i,x)}(b^k,z)=(b^k,z)$. Similarly, $L_{(b^i,x),(b^j,y)}(b^k,z)=(b^k,z)=R_{(b^i,x),(b^j,y)}(b^k,z)$. Hence $(b^k,z)\in Z(Q)$.

We also claim that for every $x\in X$ there is at most one $c\in C$ such that $(c,x)\in S$. Suppose that $(c,x),(d,x)\in S$ for some $c,d\in C$. We have $(c,x)=(c,0)(1,x)$ and $(d,x)=(d,0)(1,x)$. Hence $(cd^{-1},0) = (c,0)(d,0)^{-1}\in S$. Since $S$ intersects $C\times 0$ trivially, we have $c=d$. We conclude that $|b^k|=|z|$.
\end{proof}

\begin{theorem}[Split $3$-divisible abelian-by-cyclic Moufang loops]\label{Th:Split}
Suppose that $(X,+)$ is a $3$-divisible abelian group and $C$ a $3$-divisible cyclic group. The following conditions are equivalent:
\begin{enumerate}
\item[(i)] $Q$ is a Moufang loop, $X\unlhd Q$, $Q=CX$ and $C\cap X=1$.
\item[(ii)] $Q$ is isomorphic to $C\ltimes_{(f^3,\beta)}X$ for some Moufang permutation $f$ on $(X,+)$ with associated biadditive mapping $\beta$.
\end{enumerate}
\end{theorem}
\begin{proof}
The converse implication (ii) $\Rightarrow$ (i) is clear from Theorem \ref{Th:AbelianByCyclic}. For the direct implication, Theorem \ref{Th:HomImg} shows that $Q$ is the image of the homomorphism $\varphi:M=\langle a^3\rangle\ltimes_{(f^3,\beta)}X\to Q$, $\varphi(a^{3i},x)=a^{3i}x$, where $f$ is the restriction of $T_a$ to $X$. Suppose that $(a^{3i},x)$ is in the kernel of $\varphi$. Then $a^{3i}x=1$ implies that $x\in C\cap X=1$ and $a^{3i}=1$. Hence $\varphi$ is injective.
\end{proof}

\end{document}